\RequirePackage{fix-cm}
\documentclass[smallextended]{svjour3}       
\smartqed  
\usepackage{graphicx}
\usepackage{mathptmx}      
\usepackage[colorlinks=true]{hyperref} 
\hypersetup{
    linkcolor = blue,
    citecolor = blue
}
\usepackage{bm} 				
\usepackage{enumitem} 	
\usepackage{comment} 		
\usepackage{mathtools} 	
\mathtoolsset{centercolon} 
\usepackage{amssymb} 		
\usepackage{makecell} 		
\usepackage{multirow} 		
\usepackage{makecell} 		
\usepackage{boldline}  		
\usepackage{grffile}      		

\newcommand{\fd}{figures/}

\renewcommand{\d}[0]{\ensuremath{\operatorname{d}\!}} 
\newcommand{\im}{\mathrm{i}} 										  
\newcommand{\wh}[1]{\widehat{#1}} 								  

\newcommand{\eFD}{\wh{e}_{\rm FD}}
\newcommand{\eRK}{\wh{e}_{\rm RK}}
\newcommand{\eRKF}{\wh{e}_{\rm RK}^{\rm \,(fine)}}
\newcommand{\eRKC}{\wh{e}_{\rm RK}^{\rm \,(coarse)}}

\DeclareFontFamily{U}{mathx}{\hyphenchar\font45}
\DeclareFontShape{U}{mathx}{m}{n}{
      <5> <6> <7> <8> <9> <10>
      <10.95> <12> <14.4> <17.28> <20.74> <24.88>
      mathx10
      }{}
\DeclareSymbolFont{mathx}{U}{mathx}{m}{n}
\DeclareFontSubstitution{U}{mathx}{m}{n}
\DeclareMathAccent{\widecheck}{0}{mathx}{"71}
\DeclareMathAccent{\wideparen}{0}{mathx}{"75}

\usepackage{xcolor} 

\newcommand{\ptxt}[1]{#1}

\definecolor{darkgreen}{rgb}{0, 0.5, 0}

\makeatletter
\renewcommand\tableofcontents{%
\@starttoc{toc}%
}
\makeatother

\usepackage[us,12hr]{datetime} 

\usepackage{marginnote}
\makeatletter
\@mparswitchfalse%
\normalmarginpar
\makeatother

\begin{document}


\title{Efficient multigrid reduction-in-time for method-of-lines discretizations of linear advection\thanks{This work was performed under the auspices of the U.S. Department of Energy by Lawrence Livermore National Laboratory under Contract DE-AC52-07NA27344 (LLNL-JRNL-839789).
This work was supported in part by the U.S. Department of Energy, Office of Science, Office of Advanced Scientific Computing Research, Applied Mathematics program, and by NSERC of Canada.}
}

\author{H. De Sterck  \and
        		R. D. Falgout \and
        		O. A. Krzysik \and 
        		J. B. Schroder
}


\institute{H. De Sterck \at
				Department of Applied Mathematics, University of Waterloo, Waterloo, Ontario, Canada \\
				\email{hans.desterck@uwaterloo.ca} 
				\and
				R. D. Falgout \at 
				Center for Applied Scientific Computing, Lawrence Livermore National Laboratory, Livermore, California, USA \\
				\email{falgout2@llnl.gov}
				\and
				O. A. Krzysik \at 
				Department of Applied Mathematics, University of Waterloo, Waterloo, Ontario, Canada \\
				\email{okrzysik@uwaterloo.ca}
				\and
				J. B. Schroder \at 
				Department of Mathematics and Statistics, University of New Mexico, Albuquerque, New Mexico, USA \\
				\email{jbschroder@unm.edu}
}				

\date{Received: date / Accepted: date}

\maketitle

\begin{abstract}
	Parallel-in-time methods for partial differential equations (PDEs) have been the subject of intense development over recent decades, particularly for diffusion-dominated problems.
	It has been widely reported in the literature, however, that many of these methods perform quite poorly for advection-dominated problems.
	Here we analyze the particular iterative parallel-in-time algorithm of multigrid reduction-in-time (MGRIT) for discretizations of constant-wave-speed linear advection problems.
	We focus on common method-of-lines discretizations that employ upwind finite differences in space and Runge-Kutta methods in time.
	Using a convergence framework we developed in previous work, we prove for a subclass of these discretizations that, if using the standard approach of rediscretizing the fine-grid problem on the coarse grid, robust MGRIT convergence with respect to CFL number and coarsening factor is not possible.
	This poor convergence and non-robustness is caused, at least in part, by an inadequate coarse-grid correction for smooth Fourier modes in space-time known as characteristic components.
	We propose an alternative coarse-grid operator that provides a better correction of these modes. This coarse-grid operator is related to previous work and uses a semi-Lagrangian discretization combined with an implicitly treated truncation error correction. 
	Theory and numerical experiments show the proposed coarse-grid operator yields fast MGRIT convergence for many of the method-of-lines discretizations considered, including for both implicit and explicit discretizations of high order.
	\ptxt{Parallel results demonstrate substantial speed-up over sequential time-stepping.}
\keywords{parallel-in-time \and MGRIT \and Parareal \and hyperbolic PDE \and advection equation \and multigrid}
\subclass{65F10 \and 65M22 \and 65M55 \and 35L03}
\end{abstract}

\section{Introduction}
\label{sec:introduction}

Research efforts in the field of parallel-in-time integration over recent decades have led to the development of methods that can solve discretized time-dependent partial differential equations (PDEs) in faster wall-clock times than the classical technique of sequential time-stepping.
This has been achieved for highly non-trivial problems and application areas including
nonlinear PDEs \cite{Bolten_etal_2020,Falgout_etal_2017_nonlin}, PDE-constrained optimization \cite{Goetschel_Minion_2019,Guenther_etal_2018}, time-fractional PDEs \cite{Gaspar_Rodrigo_2017}, powergrid simulation \cite{Schroder_etal_2018}, machine learning \cite{Guenther_etal_2020}, and option pricing \cite{Bal_Maday_2002}.

An area for which relatively little advancement has been made is the solution of hyperbolic PDEs, and that of advection-dominated problems more broadly.
In this paper, we focus specifically on the multilevel parallel-in-time method of multigrid reduction-in-time (MGRIT) \cite{Falgout_etal_2014}, with most of our results also applying to the closely related two-level method of Parareal \cite{Lions_etal_2001}.
It has been widely reported in the literature that these two popular iterative parallel-in-time methods exhibit poor convergence when applied to advection-dominated problems \cite{Chen_etal_2014,Dai_Maday_2013,DeSterck_etal_2021,Dobrev_etal_2017,Gander_Vandewalle_2007,Gander_2008,Gander_Lunet_2020,Hessenthaler_etal_2018,Howse_etal_2019,Howse2017_thesis,KrzysikThesis2021,Nielsen_etal_2018,Ruprecht_Krause_2012,Ruprecht_2018,Schmitt_etal_2018,Schroder_2018,Steiner_etal_2015}.
Recently in \cite{DeSterck_etal_2022_LFA}, we showed heuristically that, in general, when using the standard approach of \textit{directly} discretizing the PDE on the coarse grid, poor convergence for advection-dominated problems is due, at least in part, to an inadequate coarse-grid correction of certain smooth Fourier modes known as characteristic components.\footnote{\label{ftno:direct-disc} In this paper, we define a \textit{direct} coarse-grid operator as one that discretizes the underlying PDE without any further considerations; this contrasts with the \textit{modified} coarse-grid operators we also consider, which are a modification of a direct discretization aiming to match the so-called ideal coarse-grid operator with increased accuracy.
We define a \textit{rediscretized} coarse-grid operator as the direct coarse-grid operator using the same discretization as employed on the fine-grid.}
It has long been known that the same issue plagues the spatial multigrid solution of steady state advection-dominated problems \cite{Brandt_1981,Brandt_Yavneh_1993,Yavneh_1998}. 
In \cite{DeSterck_etal_2022}, inspired by ideas from \cite{Yavneh_1998} to address this issue in the spatial context, we developed MGRIT coarse-grid operators \ptxt{that were used to solve (fine-grid) semi-Lagrangian discretizations of variable-wave-speed linear advection problems}. Specifically, these coarse-grid operators are of \textit{modified} type, consisting of a semi-Lagrangian step followed by a carefully chosen, and implicitly treated, \ptxt{truncation-error-based} correction that takes into account the fine-grid discretization.

\ptxt{The primary contribution of this paper is the extension of the modified semi-Lagrangian coarse-grid operators from \cite{DeSterck_etal_2022} to enable the solution of (fine-grid) method-of-lines discretizations with MGRIT instead of the (fine-grid) semi-Lagrangian discretizations considered in \cite{DeSterck_etal_2022}. This extension is important because method-of-lines schemes are much more widely used than semi-Lagrangian schemes.
More specifically, we target the solution of constant-wave-speed linear advection problems discretized (on the fine grid) with explicit or implicit Runge-Kutta methods in time and finite differences in space.
In earlier work \cite{DeSterck_etal_2021}, we developed coarse operators for the same model problems as considered here that yielded fast and scalable MGRIT convergence.
However, the approach from \cite{DeSterck_etal_2021} used a heuristically motivated optimization problem to generate the stencil of the coarse-grid operator such that its eigenvalues would match closely (in a specific and important sense) those of the ideal coarse-grid operator. 
Furthermore, the approach used non-practical calculations to determine the sparsity pattern of the coarse operator, which, along with the specific optimization approach, was only possible due to the simple nature of the PDE.
The limitations of \cite{DeSterck_etal_2021} motivate the coarse-grid operator framework we develop in this paper, which has much stronger theoretical justifications, and, can, at least in principle, be used for more complicated PDEs (even though we do not consider them in this paper).
Besides those in \cite{DeSterck_etal_2021}, the coarse-grid operators we develop here are the first that lead to fast and scalable MGRIT convergence for method-of-lines discretizations of advection problems, and they are the first to do so in a cost-effective and generalizable way.
}

The idea of using a coarse-grid semi-Lagrangian discretization (without the modification we propose in this paper) with a fine-grid 
method-of-lines discretization was first proposed in \cite{Schmitt_etal_2018} for a nonlinear advection-diffusion PDE.
However, \cite{Schmitt_etal_2018} reported substantial deterioration in convergence in the hyperbolic limit of their problem, and performance was not reported with respect to several important discretization and solver parameters.
Importantly, a consequence of \cite[Thm. 6.4]{DeSterck_etal_2022_LFA} is that the combination of direct fine- and coarse-grid discretizations in \cite{Schmitt_etal_2018} (i.e., without modification) cannot produce robust convergence in general, because even for the simple case of an explicit first-order discretization of the constant-wave-speed linear advection problem, it can result in poor convergence.\footnote{\label{ftn:ERK1+U1-SL-equiv} It can be shown that the convergence factor of the method exceeds unity for almost all CFL numbers. 
Note \cite[Thm. 6.4]{DeSterck_etal_2022_LFA} assesses MGRIT convergence when using semi-Lagrangian discretizations on both grids; however, for CFL numbers less than unity, the first-order semi-Lagrangian discretization is equivalent to using first-order upwind finite differences in space and explicit Euler in time (see, e.g., \cite[p. 359]{Durran_2010}).} 
The remainder of this paper is organized as follows.
Preliminaries are given in Section~\ref{sec:preliminaries}, including a description of the model problem and an overview of MGRIT.
In Section~\ref{sec:SL}, a description of semi-Lagrangian methods is provided \ptxt{followed by further details about the modified coarse-grid operators from \cite{DeSterck_etal_2022} that are relevant to this work}.
Section~\ref{sec:MOL} considers the method-of-lines discretizations that are the focus of this paper and their discretization errors.
Section~\ref{sec:MOL-redisc-analysis} analyzes MGRIT convergence for method-of-lines discretizations when using direct coarse-grid operators.
New modified coarse-grid operators and associated numerical results are given in Section~\ref{sec:MOL-trunc-correction}.
Concluding remarks are given in Section~\ref{sec:conclusion}.

\section{Preliminaries}
\label{sec:preliminaries}

\subsection{Model problem}
\label{sec:model-problem}

We consider the one-dimensional, constant-wave-speed linear advection problem,
\begin{align} \label{eq:ad}
\frac{\partial u}{\partial t} + \alpha \frac{\partial u}{\partial x} 
= 
0,
\quad
(x, t) \in (-1, 1) \times (0, T],
\quad
u(x, 0) = u_0(x),
\quad
\alpha > 0,
\end{align}
with $u$ subject to periodic boundary conditions in space. 
We consider both semi-Lagrangian and method-of-lines discretizations for this problem, which are described in more detail in Sections~\ref{sec:SL}~and~\ref{sec:MOL}, respectively.
Both discretizations use the same underlying space-time mesh. 
Specifically, the spatial domain is discretized with $n_x$ points equally separated by a distance of $h$: $\big\{ x_i = -1 + (i-1)h \colon i = 1, \ldots, n_x\big\}$.
We also use the notation ${\bm{x} = \big( x_1, \ldots, x_{n_x} \big)^\top \in \mathbb{R}^{n_x}}$.
The time domain is discretized using $n_t+1$ points equally separated by a distance of $\delta t$: $\big\{ t_n = n \delta t \colon n = 0, \ldots, n_t \big\}$.
On this space-time mesh, both of the aforementioned discretizations result in a fully discrete system of equations that takes the form
\begin{align} \label{eq:one-step-general}
\bm{u}_{n+1} = \Phi \bm{u}_n, \quad n = 0, \ldots, n_t - 1, \quad \bm{u}_0 = u_0(\bm{x}),
\end{align}
in which $(\bm{u}_n)_i \approx u(x_i, t_n)$, and the matrix $\Phi \in \mathbb{R}^{n_x \times n_x}$ is the \textit{time-stepping operator} responsible for propagating the spatial approximation forwards in time by an amount $\delta t$.

\subsection{Algorithmic description of MGRIT}
\label{sec:MGRIT}

MGRIT \cite{Falgout_etal_2014} is an iterative multigrid method for solving systems of equations of the form of \eqref{eq:one-step-general} in parallel over $n$.
A two-level MGRIT iteration combines fine-grid relaxation with a coarse-grid correction.

The fine grid is taken to be the set of time points underlying \eqref{eq:one-step-general}, $\big\{ t_n = n \delta t \colon n = 0, \ldots, n_t - 1 \big\}$.
A coarse grid is then induced from a coarsening factor $m \in \mathbb{N} \setminus \{1 \}$ by taking every $m$th time point from the fine grid, $\big\{ t_n = n m \delta t \colon n = 0, \ldots, \frac{n_t}{m} - 1 \big\}$, supposing for simplicity that $n_t$ is divisible by $m$.
The set of points appearing exclusively on the fine grid are denoted as ``F-points,'' and those shared between the fine and coarse grids are ``C-points.''
We dub the set of points consisting of a C-point and the $m-1$ F-points that follow it as a ``CF-interval.''
The fine-grid relaxation scheme acts in a block fashion, making use of so-called F- and C-relaxations, which set the algebraic residual to be zero at F- and C-points, respectively.
Specifically, the post-relaxation occurring after the coarse-grid correction is just an F-relaxation.
We write the pre-relaxation occurring before the coarse-grid correction as F(CF)$^{\nu}$ with $\nu \in \mathbb{N}_0$, indicating that it consists of a single F-relaxation, followed by $\nu$ sweeps of CF-relaxation---a C-relaxation followed immediately by an F-relaxation. 

The coarse-grid correction consists of solving the following system with a factor of $m$ fewer equations than the fine-grid problem \eqref{eq:one-step-general},
\begin{align} \label{eq:one-step-coarse-general}
\bm{e}_{n+1}^{( \Delta )} = \Psi \bm{e}_n^{(\Delta )} + \bm{r}_{n+1}^{(\Delta)}, \quad n = 0, \ldots, \frac{n_t}{m} - 1, \quad \bm{e}_0^{( \Delta )} = \bm{0},
\end{align}
\ptxt{where $(\Delta)$ superscripts denote coarse-grid variables.}
In the two-grid setting, problem \eqref{eq:one-step-coarse-general} is solved via sequential time-stepping, but in the multilevel context it is solved by recursively applying MGRIT, noting that it has the same structure as \eqref{eq:one-step-general}.
The system \eqref{eq:one-step-coarse-general} is an approximation to the C-point Schur complement of the residual equation of \eqref{eq:one-step-general}, with the approximation being characterized by the \textit{coarse-grid time-stepping operator} $\Psi \approx \Phi^m$ \cite{Dobrev_etal_2017,Southworth_etal_2021,DeSterck_etal_2022_LFA}.
In \eqref{eq:one-step-coarse-general}, $\bm{r}_{n}^{(\Delta)}$ is the residual at the $n$th C-point, and $\bm{e}_n^{(\Delta)}$ is an approximation to the associated algebraic error. 
If the so-called \textit{ideal coarse-grid operator} is used $\Psi = \Phi^m =: \Psi_{\rm ideal}$, then $\bm{e}_n^{(\Delta)}$ is exactly the C-point error, and MGRIT converges to the exact solution of \eqref{eq:one-step-general} in a single iteration.
However, the ability to obtain parallel speed-up with MGRIT over sequential time-stepping hinges on the solution of \eqref{eq:one-step-coarse-general} being cheaper to obtain than that of the fine-grid problem \eqref{eq:one-step-general}.
In particular, if $\Psi = \Psi_{\rm ideal}$ then (in general) no speed-up can occur since  solving the coarse-grid correction problem is just as expensive as solving the original fine-grid problem.
On the other hand, the convergence speed of MGRIT is characterized by the accuracy of the approximation $\Psi \approx \Psi_{\rm ideal}$  \cite{Dobrev_etal_2017,Southworth_etal_2021,DeSterck_etal_2022_LFA}.
The goal, therefore, is to strike a balance between the cost of applying $\Psi$ and the quality of the approximation it provides to $\Psi_{\rm ideal}$.

Most often in the PDE context, $\Psi$ is derived by rediscretizing the fine-grid problem; that is, using the same discretization with the enlarged coarse-grid time-step size $m \delta t$.
For diffusive problems, this tends to result in a quickly converging solver, but for advection-dominated problems, it does not.
The main objective of this paper is to develop an alternative to naive rediscretization, and, more generally, to any direct discretization (see Footnote \ref{ftno:direct-disc}), by carefully modifying the direct discretization, when the fine-grid operator $\Phi$ is a method-of-lines discretization of \eqref{eq:ad}.

\section{Semi-Lagrangian discretizations and the modified coarse-grid operators from \cite{DeSterck_etal_2022}}
\label{sec:SL}

\ptxt{
Here we provide a brief description of semi-Lagrangian methods for \eqref{eq:ad} because the modified coarse-grid operators we propose require an understanding of these discretizations.
Following this, we give a simplified overview of the modified semi-Lagrangian coarse-grid coarse operators from \cite{DeSterck_etal_2022}.
}
%

\begin{figure}[t!]
\centerline{
\includegraphics[scale=0.85]{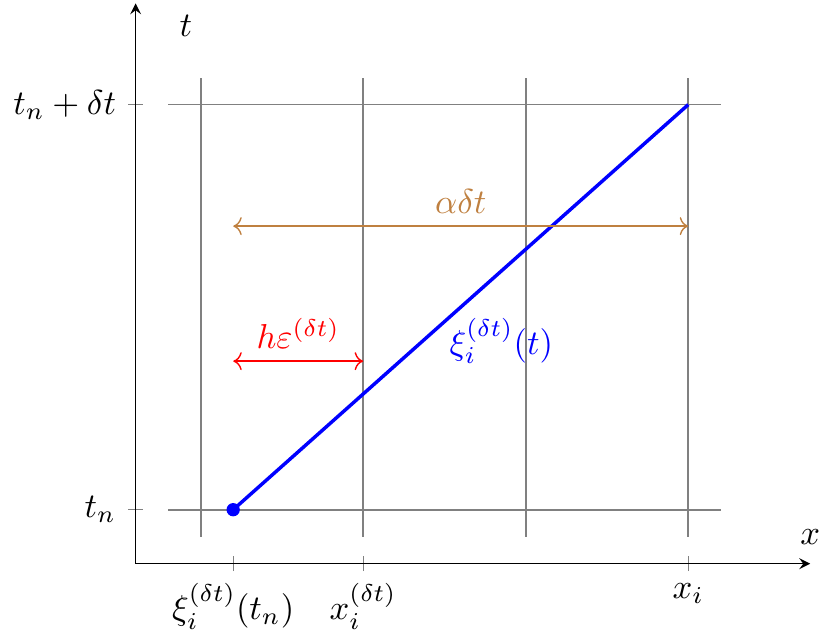}
}
\caption{
A local characteristic $\xi_{i}^{(\delta t)}(t)$ for $t \in [t_n, t_{n+1}]$ of the advection problem \eqref{eq:ad}.
By definition, the characteristic passes through the arrival point $(x,t)=(x_i, t_{n+1})$. 
The departure point (or foot) of the characteristic is its location at time $t = t_n$.
The departure point is decomposed into the sum of its east-neighboring mesh point $x_i^{(\delta t)}$ and its distance $h \varepsilon^{(\delta t)}$ from this point.
\label{fig:departure_diagram}
}
\end{figure}

We denote the time-stepping operator of a semi-Lagrangian discretization of \eqref{eq:ad} as $\Phi = {\cal S}_{p}^{(\delta t)} \in \mathbb{R}^{n_x \times n_x}$, indicating that it uses a time-step size of $\delta t$ and has a (global) order of accuracy of $p$.
To advance $\bm{u}_n$ to $\bm{u}_{n+1}$, the semi-Lagrangian method uses the fact that the solution of \eqref{eq:ad} along characteristics is constant.
We define a local characteristic of \eqref{eq:ad} that arrives at the mesh point $(x,t) = (x_i, t_n + \delta t)$, and departs from time $t = t_n$, as the curve $(x, t) = \big( \xi^{(\delta t)}_i(t), t \big)$; see Fig.~\ref{fig:departure_diagram}.

For every such local characteristic, we locate the departure point $(x,t) = \big( \xi_i^{(\delta t)}(t_n), t_n \big)$, which is easily computed as $\xi_i^{(\delta t)}(t_n) = x_i - \alpha \delta t \equiv x_i^{(\delta t)} - h \varepsilon^{(\delta t)}$.
Here, $x_i^{(\delta t)}$ denotes the mesh point immediately to the east of $\xi_i^{(\delta t)}(t_n)$, and $\varepsilon^{(\delta t)} \in [0,1)$ is the mesh-normalized distance between $\xi_i^{(\delta t)}(t_n)$ and $x_i^{(\delta t)}$.
In general, $\xi_i^{(\delta t)}(t_n)$ does not coincide with a mesh point, so the solution at this point is approximated by fitting an interpolating polynomial of degree at most $p$ through the known solution $\bm{u}_n$ at the mesh points nearest to the departure point.
The $i$th component of $\bm{u}_{n+1}$ is then given by this interpolant of $\bm{u}_n$.
Unlike its explicit method-of-lines counterparts (described in the next section), this discretization is unconditionally stable with respect to $\delta t$, i.e., $\big\Vert {\cal S}_{p}^{(\delta t)} \big\Vert_2 \leq 1$. 
See the textbooks \cite{Durran_2010,Falcone_Ferretti_2014} for a more detailed description of semi-Lagrangian methods.

\ptxt{
In \cite{DeSterck_etal_2022} we used MGRIT to solve advection problems discretized (on the fine grid) with semi-Lagrangian methods.
We showed numerically that rediscretizing a semi-Lagrangian method on a coarse grid results in extremely poor MGRIT convergence.
To overcome this, we proposed instead to use} a modified coarse-grid operator that consists of a rediscretized semi-Lagrangian step, followed by a correction step. 
The correction was carefully designed so that the dominant term in the truncation error of the modified operator matches that of the ideal coarse-grid operator.\footnote{\ptxt{For constant-wave-speed advection problems, dominant truncation errors are matched exactly, while for variable-wave-speed  problems they are matched in an approximate way.}} 
This solution was based on the idea of matching truncation errors originally proposed in \cite{Yavneh_1998} for spatial multigrid methods, and is motivated further by our analysis in \cite[Sec. 6]{DeSterck_etal_2022_LFA}.

\ptxt{
We now provide some further details on the construction of the modified coarse-grid operator from \cite{DeSterck_etal_2022} to help contextualize and motivate the coarse-grid operators for method-of-lines schemes that we develop later in Section \ref{sec:MOL-trunc-correction}.
In order to carry out the above described truncation error matching, one requires error estimates of both the ideal coarse-grid operator, which, in the case of \cite{DeSterck_etal_2022}, is $m$ steps of the fine-grid semi-Lagrangian discretization ${\cal S}_{p}^{(\delta t)}$, and the coarse-grid semi-Lagrangian discretization ${\cal S}_{p}^{(m \delta t)}$.
Both of these estimates can be derived from the following error estimate for ${\cal S}_p^{(\delta t)}$, which is a simplified version of \cite[Lem. 3.1]{DeSterck_etal_2022}. 
}
\begin{lemma}[Semi-Lagrangian truncation error]
\label{lem:SL_trunc}
Define $\bm{u}(t) \in \mathbb{R}^{n_x}$ as the vector composed of the exact solution to PDE \eqref{eq:ad} sampled on the spatial mesh at time $t$. 
Let \ptxt{the difference matrix ${\cal D}^{(p+1)}_s \in \mathbb{R}^{n_x \times n_x}$ be as described below.} 
Then, for sufficiently smooth solutions of the PDE \eqref{eq:ad}, the local truncation error of the above described semi-Lagrangian discretization can be expressed as
\begin{align} \label{eq:SL_trunc}
\bm{u}(t_{n+1}) - {\cal S}_{p}^{(\delta t)} \bm{u}(t_{n})
=
(-h)^{p+1} 
f_{p+1} \big(\varepsilon^{(\delta t)} \big) 
\frac{{\cal D}^{(p+1)}_s}{h^{p+1}}
\bm{u}(t_{n+1})
+ {\cal O}(h^{p+2}),
\end{align}
where the function $f_{p+1}$ is the following degree $p+1$ polynomial:
\begin{align} \label{eq:fpoly_def}
f_{p+1}(z) := \frac{1}{(p+1)!} \prod \limits_{j = - \ell(p)}^{r(p)} (j+z).
\end{align}
\end{lemma}
In \eqref{eq:fpoly_def}, $\ell(p), r(p) \in \mathbb{N}_0$ are such that, for a given $p$, the $p+1$ mesh points $\big\{ x_i^{(\delta t)} + j h \big\}_{j = - \ell(p)}^{r(p)}$ in the interpolation stencil are those closest to the departure point $\xi_i^{(\delta t)}(t_n)$.
Furthermore, the matrix 
\ptxt{
${\cal D}^{(p+1)}_s \in \mathbb{R}^{n_x \times n_x}$ in \eqref{eq:SL_trunc}
}
is defined such that $h^{-(p+1)}{\cal D}^{(p+1)}_s$ represents an order $s \in \mathbb{N}$ accurate finite-difference approximation of the $p+1$st derivative of a periodic grid function.
Let $\bm{v} = \big(v(x_1), \ldots, v(x_{n_x}) \big)^\top \in \mathbb{R}^{n_x}$ denote a vector representing a periodic function $v(x)$ evaluated on the spatial mesh. 
Then, if $v$ is at least $p+1+s$ times continuously differentiable,
\begin{align} \label{eq:Dp+1_def}
\left(
\frac{{\cal D}^{(p+1)}_s}{h^{p+1}} \bm{v} 
\right)_i = \left. \frac{\d^{p+1} v}{\d x^{p+1}} \right|_{x_i} + {\cal O}(h^s), 
\quad i \in \{1, \ldots, n_x \}.
\end{align}
Since the mesh points are equispaced, the matrix ${\cal D}^{(p+1)}_s$ is circulant.
Finally, note that while the entries of ${\cal D}^{(p+1)}_s$ are independent of $h$, its action is not: ${\cal D}^{(p+1)}_s \bm{v} = {\cal O}(h^{p+1})$ if $\bm{v}$ is independent of $h$.

\ptxt{With error estimates of the ideal coarse-grid operator and the coarse-grid semi-Lagrangian discretization ${\cal S}_{p}^{(m \delta t)}$, one can develop the following coarse-grid operator \cite[(3.20)]{DeSterck_etal_2022}}
\begin{align} \label{eq:SL-redisc-corrected}
\Psi 
= 
\Big( I - \gamma_{p+1}^{(\rm trunc)} {\cal D}^{(p+1)}_s \Big)^{-1} {\cal S}_p^{(m \delta t)}
\approx
\Psi_{\rm ideal}
=
\prod_{k = 0}^{m-1}
{\cal S}_p^{(\delta t)},
\end{align}
\ptxt{which has the property that the dominant term in its truncation error matches that of the ideal coarse-grid operator.
Specifically, the $\gamma_{p+1}^{(\rm trunc)}$ term captures} the difference between the constants in the leading-order truncation errors of ${\cal S}_p^{(m \delta t)}$ and $\Psi_{\rm ideal}$.
Note that the correction in \eqref{eq:SL-redisc-corrected} is done implicitly, \ptxt{i.e., it requires the inversion of a sparse matrix, for numerical stability reasons}.

\ptxt{The primary contribution of this paper is the generalization of the modified coarse-grid operator $\Psi$ in \eqref{eq:SL-redisc-corrected} so that it can be used in conjunction with fine-grid method-of-lines discretizations and not only the fine-grid semi-Lagrangian discretizations of \cite{DeSterck_etal_2022}. The generalized coarse-grid operators developed here have the same structure as \eqref{eq:SL-redisc-corrected}, in that they consist of a semi-Lagrangian step followed by a correction; however, the correction is different than in \eqref{eq:SL-redisc-corrected} because it takes into account the truncation error of the underlying fine-grid method-of-lines scheme.}

\section{Method-of-lines discretization}
\label{sec:MOL}

In this section, we outline the method-of-lines discretizations that we consider for the model problem in Section~\ref{sec:model-problem}.
\ptxt{The (fine-grid) discretizations we consider in this work are the same as those we considered previously in \cite{DeSterck_etal_2021}}, and the reader can find in \cite{DeSterck_etal_2021} specific details of these discretizations that are omitted here, \ptxt{such as finite-difference formulae and Butcher tableaux}.\footnote{With the exception of the fifth-order singly diagonally implicit Runge-Kutta method we consider here but did not in \cite{DeSterck_etal_2021}. The Butcher tableau for this method can be found in \cite[Tab. 24]{Kennedy_Carpenter2016}.}
This section also develops truncation error estimates for the discretizations, which are used in our MGRIT convergence analysis later in the paper, and for developing our new coarse-grid operators.
While developing such error estimates for constant-coefficient PDEs is in some sense a textbook exercise, we are unaware of any references that provide exactly the estimates we require. 
We do remark, however, that our analysis has similarities with the Fourier-based error analysis in \cite{Gander_Lunet_2020}.
%

\subsection{Finite-difference spatial discretization}
\label{sec:FD}

\ptxt{
We define the circulant matrix ${\cal L}_p \in \mathbb{R}^{n_x \times n_x}$ such that $\frac{{\cal L}_p}{h}$ represents a $p$th-order accurate finite-difference approximation to the first derivative on a periodic grid.
} 
The finite-difference rule that ${\cal L}_p$ encodes is given by differentiating a degree $p$ interpolating polynomial that is fit through $p+1$ contiguous mesh points \cite{Fornberg_1988,Shu_1998}.
The mesh points involved in the finite-difference stencil to approximate $\frac{\d}{\d x} |_{x_i}$ are $\{ x_i + j h \}_{j = - \ell(p)}^{r(p)}$ (with $\ell(p)$ and $r(p)$ not necessarily the same functions as those used in the semi-Lagrangian interpolation stencil from Section~\ref{sec:SL}).
We consider upwind-finite-difference rules which means the stencils have a bias to the left of $x_i$ since $\alpha > 0$ in \eqref{eq:ad}.\footnote{For even $p$, central stencils are typically more accurate than upwind stencils (see, e.g., \cite{Gander_Lunet_2020}). 
However, we use upwind stencils since they possess better stability properties when paired with explicit Runge-Kutta time integrators, and because we find that the resulting space-time discretizations are more favourable from an MGRIT convergence standpoint.}
For odd $p \geq 1$ the stencils have a one-point bias, $\ell(p) = \frac{p+1}{2}$, $r(p) = \ell(p)-1$; for even $p \geq 2$ the stencils have a two-point bias, $\ell(p) = \frac{p}{2}+1$, $r(p) = \ell(p)-2$.
We abbreviate ``$p$th-order upwind discretization'' as ``U$p$.''

The below lemma is a slight variation of a well-known error estimate for ${\cal L}_p$. 
\begin{lemma}[Finite-difference error] \label{lem:FD_trunc}
Let ${\cal D}^{(p+1)}_s \in \mathbb{R}^{n_x \times n_x}$ be an $s$th-order approximation to the $p+1$st derivative of a periodic grid function, as in \eqref{eq:Dp+1_def}. 
Let $v(x)$ be a periodic function that is at least $p+1+s$ times differentiable, and let $\bm{v} = \big(v(x_1), \ldots, v(x_{n_x}) \big)^\top \in \mathbb{R}^{n_x}$ denote the vector composed of $v$ evaluated on the spatial mesh.
Then, the error of the above described finite-difference spatial discretization $\frac{{\cal L}_p}{h}$ satisfies
\begin{align} \label{eq:L_FD_trunc}
\left. \frac{\d v}{\d x} \right|_{x_i}  
- 
\left. \left( \frac{{\cal L}_p}{h} \bm{v} \right) \right|_{i}
= 
h^{p}
\wh{e}_{\rm FD}  
\left( \frac{{\cal D}^{(p+1)}_s}{h^{p+1}} \bm{v} \right)_{i} + {\cal O}(h^{p+1}),
\quad
i \in \{1, \ldots, n_x \},
\end{align}
with constant
\begin{align} \label{eq:e_FD_def}
\wh{e}_{\rm FD} 
:= 
(-1)^{r(p)} \frac{\ell(p)! r(p)!}{(p+1)!}.
\end{align}
\end{lemma}

\begin{proof}
Consider interpolating $v(x)$ at the $p+1$ mesh points $\{ x_i + j h \}_{j = - \ell(p)}^{r(p)}$ using a polynomial of degree at most $p$.
Since $v$ is sufficiently smooth, the standard error estimate from polynomial interpolation theory (see, e.g., \cite[Thm. 3.1.1]{Davis_1975}) can be used to quantify the difference between $v$ and this polynomial for $x \in (x_i - \ell h, x_i + r h)$.
It is not difficult to show (see, e.g., \cite[pp. 165--166]{Hesthaven_2017}) that 
differentiating this estimate and evaluating the result at $x = x_i$ gives
\begin{align} \label{eq:L_FD_trunc_proof_aux}
\left. \frac{\d v}{\d x} \right|_{x_i} 
- 
\left. \left( \frac{{\cal L}_p}{h} \bm{v} \right) \right|_{i}
= 
h^p \wh{e}_{\rm FD} 
\left. \frac{\d^{p+1} v}{\d x^{p+1}} \right|_{\xi_i},
\end{align}
with the point $\xi_i \in ( x_i - \ell h, x_i + r h )$ not known.
Since $\xi_i$ and $x_i$ are a distance of ${\cal O}(h)$ apart, by Taylor expansion, we have 
$
\left. \frac{\d^{p+1} v}{\d x^{p+1}} \right|_{\xi_i}
=
\left. \frac{\d^{p+1} v}{\d x^{p+1}} \right|_{x_i}
+
{\cal O}(h)
$.
Plugging this into \eqref{eq:L_FD_trunc_proof_aux} and replacing the $p+1$st derivative of $v$ with an approximation using ${\cal D}^{(p+1)}_s$ from \eqref{eq:Dp+1_def} gives the claim \eqref{eq:L_FD_trunc}. 
\qed
\end{proof} 

Results later in the paper require the error constants \eqref{eq:e_FD_def}, so, for reference purposes they can be found in Table~\ref{tab:e_FD_and_RK} for the discretizations considered here.
%

\renewcommand*{\arraystretch}{1.3} 
\begin{table}[b!] 
\caption{
Left: Error constants $\wh{e}_{\rm FD}$ defined in \eqref{eq:e_FD_def} for order $p$ upwind spatial discretizations, U$p$. 
Right: Error constants $\eRK$ defined in \eqref{eq:e_RK_def} for order $q$ Runge-Kutta methods, ERK$q$ and SDIRK$q$. 
\label{tab:e_FD_and_RK}
}
\begin{center}
\begin{tabular}{|c|c|} 
%
%
\hline
$p$ & $\wh{e}_{\rm FD}$: U$p$ 								\\ \hline
1 & $\hphantom{-}5 \times 10^{-1}$                      	\\ \hline
2 & $\hphantom{-}3.3333\ldots \times 10^{-1}$    	\\ \hline
3 & $-8.3333\ldots \times 10^{-2}$                       	\\ \hline
4 & $-5 \times 10^{-2}$                                          \\ \hline
5 & $\hphantom{-}1.6667\ldots \times 10^{-2}$    \\ \hline
\end{tabular}
\quad
%
%
\begin{tabular}{|c|c|c|c|c|c|} 
\hline
$q$ & $\eRK$: ERK$q$ & $\eRK$: SDIRK$q$ 															    \\ \hline
1 & $-5 \times 10^{-1}$					  & $\hphantom{-}5 \times 10^{-1}$               		\\ \hline
2 & $-1.6667 \ldots \times 10^{-1}$ & $\hphantom{-}4.0440 \ldots \times 10^{-2}$	\\ \hline
3 & $-4.1667 \ldots \times 10^{-2}$ & $-2.5897 \ldots \times 10^{-2}$ 						\\ \hline
4 & $-8.3333 \ldots \times 10^{-3}$ & $-8.4635 \ldots \times 10^{-4}$ 						\\ \hline
5 & $-6.0764 \ldots \times 10^{-4}$ & $\hphantom{-}5.3005 \ldots \times 10^{-4}$ 	\\ \hline
\end{tabular}
\end{center}
\end{table}

\subsection{Runge-Kutta temporal discretization}
\label{sec:RK}

We consider both explicit Runge-Kutta (ERK) and singly diagonally implicit Runge-Kutta (SDIRK) methods with global order of accuracy $q \in \mathbb{N}$, with the resulting methods being denoted by ERK$q$ and SDIRK$q$, respectively.
Note that ERK1 and SDIRK1 correspond to the explicit and implicit Euler methods, respectively.
The \textit{stability function} for such Runge-Kutta methods is the rational function $R_q(z) 
= 
\frac{P(z)}{Q(z)}$,
in which $P$ and $Q$ are polynomials deriving from the Butcher tableau of a given method, with $Q$ being the identity in the case of an ERK method.
The first $q+1$ terms in the Taylor expansion of $R_q(z)$ match that of $\exp(z)$ (see, e.g., \cite{Butcher2003}),
\begin{align} \label{eq:RK_stab_def}
R_q(z) 
&= 
\sum \limits_{j = 0}^{q} \frac{z^j}{j!} + \sum \limits_{j = q+1}^{\infty} \beta_j z^j, \quad \beta_{q+1} \neq \frac{1}{(1+q)!},
\\
&=
\label{eq:e_RK_def}
\exp(z) + \eRK z^{q+1} + {\cal O}(z^{q+2}),
\quad \eRK := \beta_{q+1} - \frac{1}{(q+1)!}.
\end{align}
That is, $\beta_{q+1}$ is the first coefficient in the Taylor expansion of $R_q(z) $ that differs from the coefficients in the Taylor expansion of the exponential.
The constant $\eRK$ is defined to capture the dominant error in the Runge-Kutta approximation.
Values of the error constants $\eRK$ for the Runge-Kutta methods we consider in this work are required later, and they can be found in Table~\ref{tab:e_FD_and_RK}.

\subsection{Method-of-lines discretization and its truncation error}
\label{sec:MOL-trunc-analysis}

We now combine the spatial and temporal discretizations from Section~\ref{sec:FD} and Section~\ref{sec:RK}.
Discretizing the advection problem \eqref{eq:ad} in space using the finite-difference method from Section~\ref{sec:FD}, we arrive at the system of ODEs
\begin{align}
\frac{\d \bm{v}}{\d t} = - \alpha \frac{{\cal L}_p }{h} \bm{v},
\quad
t \in (0, T],
\quad
\bm{v}(0) = \bm{u}_0,
\end{align}
in which $v_i(t) \approx u(x_i, t)$ is the approximation to the solution of PDE \eqref{eq:ad} at the $i$th spatial grid point.
Next these ODEs are discretized using the Runge-Kutta method from Section~\ref{sec:RK}, writing the fully discrete approximation as $\bm{u}_{n} \approx \bm{v}(t_n)$; that is, $(\bm{u}_n)_i \approx u(x_i, t_n)$.
Since ${\cal L}_p$ is diagonalizable, the fully discrete scheme may be written \ptxt{in the one-step form}\footnote{\ptxt{Writing ${\cal M}^{(\delta t)}_{p, q}$ in terms of the rational function $R_q$ is for theoretical purposes only. 
In practice, a time-step with ${\cal M}^{(\delta t)}_{p, q}$ is carried out in the standard way for a Runge-Kutta method: By a sequence of matrix-vector products with ${\cal L}_p$ in the case of an ERK method, and for an SDIRK method through a sequence of linear solves with matrices of the form $I - a {\cal L}_p$, for constant $a$, which we perform with SuiteSparse's UMFPACK \cite{Davis2004}.}} (see, e.g., \ptxt{the derivation in} \cite[Lem. 1]{DeSterck_etal_2021})
\begin{align} \label{eq:MOL_M_def}
\bm{u}_{n+1} 
= 
R_q \bigg(- \alpha \delta t \frac{{\cal L}_p }{h} \bigg) \bm{u}_n 
\equiv
{\cal M}^{(\delta t)}_{p, q} \bm{u}_n,
\quad
n = 0, 1, \ldots, n_t -1,
\end{align}
where we have defined the time-stepping operator for the scheme as ${\cal M}^{(\delta t)}_{p, q} \in \mathbb{R}^{n_x \times n_x}$.
In an attempt to balance spatial and temporal discretization errors we consider only space-time discretizations with $q = p$ (with the exception of the theoretical result in Lemma~\ref{lem:MOL-trunc}). 
We denote the explicit and implicit discretizations by ERK$p$+U$p$ and SDIRK$p$+U$p$, respectively.
The ERK$p$+U$p$ methods have a CFL limit, meaning that a necessary condition for stable time integration is that $\delta t \leq c_{\rm max} h$, with $c := \frac{\alpha \delta t}{h} \in (0, c_{\rm max}]$ the CFL number.
For reference purposes, values of $c_{\rm max}$ for the discretizations we consider in this paper can be found in Table~\ref{tab:CFL_limits}.
All SDIRK$p$+U$p$ schemes we consider are unconditionally stable since the SDIRK methods are A-stable and the eigenvalues of the $-{\cal L}_{p}$ lie in the negative half-plane.
\begin{table}[t!] 
\caption{
CFL limits $c_{\rm max}$ for ERK$p$+U$p$ discretizations of the advection problem \eqref{eq:ad}. 
\label{tab:CFL_limits}
}
\begin{center}
\begin{tabular}{|c|c|c|c|c|c|} 
\hline
Scheme & 
ERK1+U1  &
ERK2+U2 & 
ERK3+U3 & 
ERK4+U4 & 
ERK5+U5 \\ \hline
$c_{\rm max}$ & 
1 & $\frac{1}{2}$ 
& 1.62589$\ldots$ 
& 1.04449$\ldots$
& 1.96583$\ldots$ \\ \hline
\end{tabular}
\end{center}
\end{table}

We conclude this section with Lemma~\ref{lem:MOL-trunc} below estimating the local truncation error of a method-of-lines discretization, followed by Definition~\ref{def:diss_vs_disp} classifying the discretizations in terms of their numerical error.
\begin{lemma}[Method-of-lines truncation error] \label{lem:MOL-trunc}
Define $\bm{u}(t) \in \mathbb{R}^{n_x}$ as the vector composed of the exact solution to PDE \eqref{eq:ad} sampled on the spatial mesh at time $t$. 
Let $c = \frac{\alpha \delta t}{h}$ be the CFL number of the discretization.
Then, 
if the solution to \eqref{eq:ad} is sufficiently smooth, the local truncation error of the above described method-of-lines discretization can be expressed as
\begin{align} \label{eq:Mpq_trunc}
\begin{aligned}
&\bm{u}(t_{n+1}) - {\cal M}^{(\delta t)}_{p, q} \bm{u}(t_n) 
=
{\cal O} \Big(
h^{p+1} \delta t,
h \delta t^{q+1}, 
\delta t^{q+2} 
\Big)
\\&
\quad\quad
-
\left[
c  
\mkern 1mu 
\eFD h^{p+1} 
\frac{{\cal D}^{(p+1)}_s}{h^{p+1}} 
+
(-c)^{q+1} \eRK h^{q+1} 
\frac{{\cal D}^{(q+1)}_s}{h^{q+1}} 
\right] 
\bm{u}(t_{n+1}),
\end{aligned}
\end{align}
where the constants $\eFD$ and $\eRK$ are defined in \eqref{eq:e_FD_def} and \eqref{eq:e_RK_def}, respectively.
\end{lemma}

\begin{proof}
See Appendix~\ref{app:MOL-trunc-proof}.
\qed
\end{proof}

Based on the largest term on right-hand side of \eqref{eq:Mpq_trunc} in Lemma~\ref{lem:MOL-trunc}, we use the following (standard) classification for the numerical error of a method-of-lines discretization.
\begin{definition}[Dissipation versus dispersion] \label{def:diss_vs_disp}
Suppose the dominant term in the truncation error \eqref{eq:Mpq_trunc} of a stable method-of-lines discretization contains the matrix ${\cal D}^{(\xi+1)}_s$ with $\xi \in \mathbb{N}$.
Then, if $\xi$ is odd, the discretization is called dissipative, and if $\xi$ is even, the discretization is called dispersive.
\end{definition}

\section{On the inadequacy of direct coarse-grid operators}
\label{sec:MOL-redisc-analysis}

In this section we analyze MGRIT convergence for method-of-lines discretizations when using direct coarse-grid operators (i.e., discretizations without corrections to match the truncation error of the ideal coarse-grid operator). 
The framework used to analyze MGRIT convergence is described in Section~\ref{sec:MOL-redisc-conv-framework}, theoretical results are presented in Section~\ref{sec:MOL-redisc-theory}, and numerical results are given in Section~\ref{sec:MOL-redisc-numerical}.

When $\Phi = $ ERK$p$+U$p$ we do not consider rediscretized coarse-grid operators, since for \textit{realistic} fine-grid CFL numbers $c$ (i.e., $c$ not much smaller than $c_{\max}$) and MGRIT coarsening factors $m$, such a coarse-grid operator is unstable, with $mc \gg c_{\rm max}$.
That is, in practice, CFL-limited method-of-lines discretizations are typically used with CFL numbers close to their limits; for example, any choice from around $c = 0.3 c_{\rm max}$, up to $c = 0.95 c_{\rm max}$ would be typical. 
Additionally, to ensure the MGRIT coarse-grid problem is sufficiently inexpensive compared to the fine-grid problem, one would typically need to use a coarsening factor $m \gg 1$ (or in the multilevel setting, the cumulative coarsening factor on the coarsest levels would need to be $\gg 1$).
We do consider, however, combining an explicit fine-grid discretization with an implicit coarse-grid discretization or a semi-Lagrangian discretization since both of these coarse-grid operators are stable. 
%

\subsection{Convergence framework}
\label{sec:MOL-redisc-conv-framework}

Let $\lambda(\omega)$ and $\mu(\omega)$ denote the Fourier symbols/eigenvalues of $\Phi$ and $\Psi$, respectively, associated with a spatial Fourier mode with frequency $\omega$.
For a problem with $n_x$ spatial degrees-of-freedom, $\omega$ discretely samples the interval $[-\pi, \pi)$ with $n_x$ points equally spaced by a distance $h$; the smoothest modes on a given spatial mesh have frequency $\omega = {\cal O}(h)$, and we dub them as \textit{asymptotically smooth}.
Let $\theta$ be a continuous temporal Fourier frequency spanning an interval of length $2 \pi$, and let $\Theta^{\rm low} := \big[ -\frac{\pi }{m}, \frac{\pi}{m} \big)$ be the space of low frequencies. 
Then, \ptxt{let $\widehat{{\cal E}} (\omega, \theta)$ be the two-level MGRIT error propagator for a space-time Fourier mode with frequency $(\omega, \theta)$, and consider the following two-level MGRIT convergence estimate for this mode:}
\begin{align} \label{eq:rho-omega-theta}
\rho \big( \widehat{{\cal E}} (\omega, \theta) \big) 
= 
\big| \lambda(\omega) \big|^{m \nu} \frac{\big| \lambda^m(\omega) - \mu(\omega) \big|}{\big| 1 - e^{- \mathrm{i} m \theta} \mu(\omega) \big|}, \quad (\omega, \theta) \in [- \pi, \pi) \times \Theta^{\rm low}.
\end{align}
Recall from Section~\ref{sec:MGRIT} that $\nu \in \mathbb{N}_0$ is the number of fine-grid CF relaxation sweeps.
For a given mode $(\omega, \theta)$, \eqref{eq:rho-omega-theta} is the infinite grid ($n_t \to \infty$) asymptotic convergence factor of two-level MGRIT as predicted by local Fourier analysis (LFA) \cite{DeSterck_etal_2022_LFA}.
Note that it is not necessary to consider $\rho \big( \widehat{{\cal E}} (\omega, \theta) \big)$ for $\theta \notin \Theta^{\rm low}$, since such modes alias with low-frequency modes on the coarse grid.

In addition to \eqref{eq:rho-omega-theta}, we are interested in the (overall) two-level convergence factor, which is the worst-case convergence factor over all space-time Fourier modes,
\begin{align} \label{eq:rho-asym}
\rho( {\cal E} ) 
:= 
\max \limits_{(\omega, \theta) \in [- \pi, \pi) \times \Theta^{\rm low}}
\rho \big( \widehat{{\cal E}} (\omega, \theta) \big) 
=
\max \limits_{\omega \in [- \pi, \pi)}
\big| \lambda(\omega) \big|^{m \nu} \frac{\big| \lambda^m(\omega) - \mu(\omega) \big|}{1 - \big|  \mu(\omega) \big|}.
\end{align}
In \eqref{eq:rho-asym}, the maximum over $\theta$ has been computed analytically, and for a given pair of discretizations $\Phi$ and $\Psi$, we will  estimate the maximum over $\omega$ by discretely sampling $[-\pi, \pi)$ with $2^{11}$ points.\footnote{Note that $\omega = 0$ is omitted from these calculations since the associated eigenvalue is unity for the discretizations we consider. 
In addition, we often omit a small number (no more than 10) of discrete frequencies closest to $\omega = 0$, since the numerical proximity of the corresponding eigenvalues to the unit circle causes catastrophic cancellation and rounding issues in the evaluation of the estimate, particularly for higher-order discretizations.}
Note that \eqref{eq:rho-asym} also bounds the coarse-grid error reduction on the first MGRIT iteration \cite{Dobrev_etal_2017}. 
Finally, note that while \eqref{eq:rho-asym} is an infinite-grid convergence estimate, it provides accurate information about the effective MGRIT convergence factor on finite-length time intervals \cite{Dobrev_etal_2017,DeSterck_etal_2022_LFA}, which is why we consider it here.

\subsection{Theoretical results for asymptotically smooth characteristic components}
\label{sec:MOL-redisc-theory}

In Theorem~\ref{thm:MOL-rho-lwr-bnd} below, we bound the convergence factor \eqref{eq:rho-asym} from below by analyzing \eqref{eq:rho-omega-theta} for \textit{asymptotically smooth characteristic components}, which are space-time modes satisfying $\theta \approx - \frac{\omega \alpha \delta t}{h}$, with $\omega = {\cal O}(h)$ and $\theta = {\cal O}(\delta t)$.
The lower bound will imply that two-grid convergence is poor due, at least in part, to the poor convergence of asymptotically smooth characteristic components.
In \cite{DeSterck_etal_2022_LFA} we showed heuristically that MGRIT convergence is not robust (with respect to problem and algorithmic parameters) when using a direct coarse-grid operator for an advection-dominated problem because asymptotically smooth characteristic components receive a poor coarse-grid correction, especially relative to all other asymptotically smooth modes.

The novel contribution of Theorem~\ref{thm:MOL-rho-lwr-bnd} is quantifying this heuristic rigorously for a subclass of method-of-lines discretizations.
Specifically, Theorem~\ref{thm:MOL-rho-lwr-bnd} is analogous to \cite[Thms. 6.1 \& 6.2]{DeSterck_etal_2022_LFA}, which used the same arguments when semi-Lagrangian discretizations are used on both the fine and coarse grids.
\begin{theorem} \label{thm:MOL-rho-lwr-bnd}
Suppose $\Phi = {\cal M}_{p,p}^{(\delta t, {\rm fine})}$ and $\Psi = {\cal M}_{p,p}^{(m \delta t, {\rm coarse})}$, with Runge-Kutta error constants \eqref{eq:e_RK_def} denoted by $\eRKF$ and $\eRKC$, respectively.
Suppose $\Phi$ and $\Psi$ use the same finite-difference spatial discretization, but possibly different Runge-Kutta methods---hence the addition of ``fine'' and ``coarse'' superscripts.
Finally, suppose that $p$ is odd.\footnote{The distinction between even and odd $p$ in the analysis in this theorem is whether the dominant truncation error term in Fourier space is real or imaginary (see also Definition~\ref{def:diss_vs_disp}).
For simplicity we only consider odd $p$, but note that analogous results can be derived for even $p$, although they have some additional subtleties (see \cite[Rem. 6.3]{DeSterck_etal_2022_LFA}). Numerical results later in the paper will consider both even and odd $p$, with a focus on odd $p$.}
Let $c = \frac{\alpha \delta t}{h}$ be the fine-grid CFL number, and $m$ the coarsening factor.
Then, the two-level MGRIT convergence factor \eqref{eq:rho-asym} satisfies the following lower bound independent of the number of CF-relaxations $\nu$
\begin{align} \label{eq:rho-lwr-bnd-char}
\rho({\cal E})
&\geq
\rho \bigg(
\widehat{{\cal E}} \bigg(\omega, \theta = - \frac{\omega \alpha \delta t}{h} \bigg)
\bigg) 
\bigg|_{\omega = {\cal O}(h)}
\\
\label{eq:rho-lwr-bnd}
&=
\widecheck{\rho}_p ( c )
\left( 
1 + 
{\cal O}
\left(h, 
\delta t, 
\delta t 
\left(\frac{\delta t}{h} \right)^{p}, 
\delta t \left( \frac{\delta t}{h} \right)^{p+1} \right) \right),
\end{align}
where the function $\widecheck{\rho}_p ( c )$ is
\begin{align} \label{eq:rho-check-MOL}
\widecheck{\rho}_p ( c )
:=
c^p
\left|
\frac{\eRKF - m^p \, \eRKC}
{\eFD + (mc)^p \, \eRKC}
\right|.
\end{align}
In addition, asymptotically in the coarse-grid CFL number $m c$, $\widecheck{\rho}_p ( c )$ satisfies
\begin{align} \label{eq:rho-check-MOL-asym}
\widecheck{\rho}_p ( c )
=
\begin{cases}
\displaystyle{
\left| 1 - \frac{1}{m^p} \frac{\eRKF}{\eRKC} \right| 
+ 
{\cal O} \left( \frac{1}{(mc)^p} \right),
}
\quad
&m c \gg 1,
\\[3ex]
{\cal O}\big( (mc )^p \big),
\quad
&m c \ll 1.
\end{cases}
\end{align}
\end{theorem}

\begin{proof}
From \eqref{eq:rho-asym}, it is clear that the inequality \eqref{eq:rho-lwr-bnd-char} holds, since the asymptotically smooth characteristic components $\theta = - \frac{\omega \alpha \delta t}{h}$ with $\omega = {\cal O}(h)$ represent only a subset of all modes over which the maximum is taken.

The next part of this proof establishes \eqref{eq:rho-lwr-bnd} by developing eigenvalue estimates for the fine- and coarse-grid discretizations, and then evaluating $\rho\big( 
\wh{{\cal E}}(\omega, \theta)
\big )$ at the aforementioned asymptotically smooth characteristic components.
In some ways, these steps mimic those used in \cite[Thms. 6.1 \& 6.2]{DeSterck_etal_2022_LFA} for semi-Lagrangian discretizations.

To estimate the eigenvalues of $\Phi$ and $\Psi$, we transform their truncation error estimates from Lemma~\ref{lem:MOL-trunc} into Fourier space where they can be interpreted as eigenvalue estimates.
Further detail about how to do this, albeit in the context of semi-Lagrangian discretizations, can be found in the proof of \cite[Thm. B.3]{DeSterck_etal_2022_LFA}.
For the fine-grid discretization $\Phi = {\cal M}_{p,p}^{(\delta t, {\rm fine})}$, the truncation error statement \eqref{eq:Mpq_trunc} gives the following eigenvalue estimate for asymptotically smooth modes $\omega = {\cal O}(h)$ when $p$ is odd:
\begin{align} \label{eq:eig-MOL-fine-est}
\begin{split}
\lambda(\omega) =
\exp\left( - \frac{\im \omega \alpha \delta t}{h} \right)
\bigg[
1 
&+ 
(-1)^{\frac{p+1}{2}} c
\left(
\wh{e}_{\rm FD} + c^p 
\mkern 1mu 
\eRKF
\right) \omega^{p+1}
\\
&\quad\quad+
{\cal O}
\big( 
\omega^{p+2}, 
\omega^{p+1} \delta t, 
\omega \delta t^{p+1}, 
\delta t^{p+2} \big)
\bigg].
\end{split}
\end{align}
Raising $\lambda(\omega)$ to the power $m$ and applying the binomial expansion gives 
\begin{align} \label{eq:eig-MOL-ideal-est}
\begin{split}
\lambda^m(\omega) =
\exp\left( - \frac{\im \omega \alpha m \delta t}{h} \right)
\bigg[
1 
&+ 
(-1)^{\frac{p+1}{2}} m c
\left(
\wh{e}_{\rm FD} 
+ 
c^p 
\mkern 1mu 
\eRKF
\right) \omega^{p+1}
\\
&\quad\quad+
{\cal O}
\big( 
\omega^{p+2}, 
\omega^{p+1} \delta t, 
\omega \delta t^{p+1}, 
\delta t^{p+2} \big)
\bigg].
\end{split}
\end{align}
Now consider the eigenvalue estimate for the coarse-grid operator $\Psi = {\cal M}_{p,p}^{(m \delta t, {\rm coarse})}$. 
This estimate takes the same form as that of \eqref{eq:eig-MOL-fine-est} except the CFL number is $mc$ rather than $c$:
\begin{align} \label{eq:eig-MOL-redisc-est}
\begin{split}
\mu(\omega) =
\exp\left( - \frac{\im \omega \alpha m \delta t}{h} \right)
\bigg[
1 
&+ 
(-1)^{\frac{p+1}{2}} m c
\left(
\wh{e}_{\rm FD} 
+ 
(mc)^p 
\mkern 1mu 
\eRKC
\right) \omega^{p+1}
\\
&\quad\quad+
{\cal O}
\big( 
\omega^{p+2}, 
\omega^{p+1} \delta t, 
\omega \delta t^{p+1}, 
\delta t^{p+2} \big)
\bigg].
\end{split}
\end{align}

Using \eqref{eq:eig-MOL-ideal-est} and \eqref{eq:eig-MOL-redisc-est} we can now estimate the convergence factor \eqref{eq:rho-omega-theta} for asymptotically smooth modes, $\omega = {\cal O}(h)$.
Considering the pre-relaxation factor $| \lambda(\omega) |^{m \nu}$ in \eqref{eq:rho-omega-theta}, we have from \eqref{eq:eig-MOL-ideal-est} that
\begin{align} \label{eq:lwr-bnd-proof-relax-est}
\big| \lambda(\omega) \big|^{m \nu}
=
\bigg|
\exp\bigg(- \frac{\im \omega \alpha m \delta t}{h} \bigg)
\bigg|^{\nu}
\big| 1 + {\cal O}(\omega^{p+1}) \big|^{\nu}
=
1 + {\cal O}(h^{p+1}).
\end{align}
(Note that the interpretation of this result is that relaxation has essentially no impact on the convergence of asymptotically smooth modes; see also \cite[Sec. 6.1]{DeSterck_etal_2022_LFA}.)
Next, from \eqref{eq:eig-MOL-ideal-est} and \eqref{eq:eig-MOL-redisc-est}, we have the following estimate for the numerator of the fraction in \eqref{eq:rho-omega-theta}:
\begin{align} 
&\big| \lambda^m(\omega) - \mu(\omega) \big|
=
\bigg|
\exp\left( - \frac{\im \omega \alpha m \delta t}{h} \right)
\bigg[
(-1)^{\frac{p+1}{2}} 
m c^{p+1}
\Big(
\eRKF
-
m^p \eRKC
\Big) 
\omega^{p+1}
\\
\notag
&\hspace{20ex}
+
{\cal O}
\big( 
\omega^{p+2}, 
\omega^{p+1} \delta t, 
\omega \delta t^{p+1}, 
\delta t^{p+2} \big)
\bigg]
\bigg|,
\\
\label{eq:lwr-bnd-proof-numer-est}
&\quad\quad\quad=
m c^{p+1}
\Big|
\eRKF
-
m^p \eRKC
\Big| 
\omega^{p+1}
\Big[1 + 
{\cal O}
\Big( 
\omega, 
\delta t, 
\delta t \left( \tfrac{\delta t}{\omega} \right)^p,
\delta t \left( \tfrac{\delta t}{\omega} \right)^{p+1} 
\Big)
\Big].
\end{align}
Finally, consider the denominator of the fraction in \eqref{eq:rho-omega-theta} for the characteristic modes $\theta = - \frac{\omega \alpha \delta t}{h}$, $\omega = {\cal O}(h)$:
\begin{align} 
&\Big|
1 - e^{- \im m \theta}\mu(\omega)
\Big|_{\theta = - \frac{\omega \alpha \delta t}{h}}
=
\Big|
- (-1)^{\frac{p+1}{2}} m c
\left(
\wh{e}_{\rm FD} + (mc)^p \eRKC
\right) 
\omega^{p+1}
\\ \notag
&\hspace{20ex}
+
{\cal O}
\big( 
\omega^{p+2}, 
\omega^{p+1} \delta t, 
\omega \delta t^{p+1}, 
\delta t^{p+2} \big)
\Big|,
\\
\label{eq:lwr-bnd-proof-denom-est}
&=
m c
\Big|
\wh{e}_{\rm FD} + (mc)^p \eRKC
\Big|
\omega^{p+1}
\Big[1 + 
{\cal O}
\Big( 
\omega, 
\delta t, 
\delta t \left( \tfrac{\delta t}{\omega} \right)^p,
\delta t \left( \tfrac{\delta t}{\omega} \right)^{p+1} 
\Big)
\Big].
\end{align}

Plugging \eqref{eq:lwr-bnd-proof-relax-est}, \eqref{eq:lwr-bnd-proof-numer-est} and \eqref{eq:lwr-bnd-proof-denom-est} into \eqref{eq:rho-omega-theta} gives the estimate
\begin{align}
\notag
&\rho \bigg(
\widehat{{\cal E}} \bigg(\omega, \theta = - \frac{\omega \alpha \delta t}{h} \bigg)
\bigg) 
\bigg|_{\omega = {\cal O}(h)}
=
\frac{\Big[1 + 
{\cal O}
\Big( 
\omega, 
\delta t, 
\delta t \left( \tfrac{\delta t}{\omega} \right)^p,
\delta t \left( \tfrac{\delta t}{\omega} \right)^{p+1} 
\Big)
\Big]}{\Big[1 + 
{\cal O}
\Big( 
\omega, 
\delta t, 
\delta t \left( \tfrac{\delta t}{\omega} \right)^p,
\delta t \left( \tfrac{\delta t}{\omega} \right)^{p+1} 
\Big)
\Big]} \times 
\\
&c^p
\left|
\frac{
\eRKF
-
m^p \eRKC
}
{
\eFD + (mc)^p \eRKC 
}
\right|
=
\widecheck{\rho}_p(c)
\Big[1 + 
{\cal O}
\Big( 
\omega, 
\delta t, 
\delta t \left( \tfrac{\delta t}{\omega} \right)^p,
\delta t \left( \tfrac{\delta t}{\omega} \right)^{p+1} 
\Big)
\Big],
\end{align}
where the final equality follows from the geometric expansion $[1 + 
{\cal O}(\epsilon)]^{-1} = [1 + 
{\cal O}(\epsilon)]$.
The claimed result \eqref{eq:rho-lwr-bnd} follows by substituting $\omega = {\cal O}(h)$ into the final equation above.

Now we show that the asymptotic relations in \eqref{eq:rho-check-MOL-asym} hold.
Rearranging \eqref{eq:rho-check-MOL},
\begin{align} \label{eq:rho-check-rewrite}
\widecheck{\rho}_p ( c )
:=
c^p
\left|
\frac{\eta - m^p}
{\xi + (mc)^p}
\right|,
\quad
\eta := \frac{\eRKF}{\eRKC}, 
\quad
\xi := \frac{\eFD}{\eRKC}.
\end{align}
Now consider the denominator of the above fraction in the two asymptotic regimes:
\begin{align} \label{eq:rho-check-denominator}
\frac{1}
{\xi + (mc)^p}
=
\begin{cases}
\frac{1}{\xi} \Big[ 1 + {\cal O}\big( (mc)^p \big) \Big], \quad &mc \ll 1,
\\
(mc)^{-p} \Big[ 1 + {\cal O} \big( (mc)^{-p} \big) \Big], \quad &mc \gg 1.
\end{cases}
\end{align}
First consider the $mc \ll 1$ case; plugging \eqref{eq:rho-check-denominator} into \eqref{eq:rho-check-rewrite} gives
\begin{align}
c^p
\left|
\frac{\eta - m^p}
{\xi + (mc)^p}
\right|
=
\frac{c^p}{|\xi|}
\left|
\big(\eta - m^p \big)
\Big[ 1 + {\cal O}\big( (mc)^p \big) \Big]
\right|
=
{\cal O} \big(c^p, (mc)^p \big)
=
{\cal O} \big( (mc)^p \big).
\end{align}
%
%
Next, consider the $m c \gg 1$ case; plugging \eqref{eq:rho-check-denominator} into \eqref{eq:rho-check-rewrite} gives
\begin{align}
c^p
\left|
\frac{\eta - m^p}
{\xi + (mc)^p}
\right|
=
\frac{c^p}{c^p}
\left|
\frac{\big(\eta - m^p \big)}{m^p}
\Big[ 1 + {\cal O} \big( (mc)^{-p} \big) \Big]
\right|
=
\left|
1 - 
\frac{\eta}{m^p}
\right|
\Big[ 1 + {\cal O} \big( (mc)^{-p} \big) \Big].
\end{align}
Expanding the last equation gives the result in \eqref{eq:rho-check-MOL-asym}.
\qed
\end{proof}

Recall from our earlier discussions that practically relevant coarse-grid CFL numbers $mc$ satisfy $mc \gg 1$. 
From Theorem~\ref{thm:MOL-rho-lwr-bnd}, and specifically \eqref{eq:rho-lwr-bnd} and \eqref{eq:rho-check-MOL-asym}, we see that MGRIT convergence of asymptotically smooth characteristic components is fast for the non-practically relevant case of $m c \ll 1$, while it is worse for $mc \gg 1$.
Specifically, for $m c \gg 1$, the MGRIT convergence factor of asymptotically smooth characteristic components is approximately $\Big| 1 - m^{-p} \frac{\eRKF}{\eRKC} \Big|$, which quickly approaches unity as $m$ increases.
Note that if the same Runge-Kutta method is used on both grids, then this approximate convergence factor simplifies to $1 - m^{-p}$.

Theorem~\ref{thm:MOL-rho-lwr-bnd} does not cover the situation in which a direct coarse-grid semi-Lagrangian method is paired with an explicit or implicit method-of-lines discretization on the fine grid.
However, similar arguments can be used to show that asymptotically smooth characteristic components receive a poor correction.
Also recall the example discussed in Section~\ref{sec:introduction} (see Footnote~\ref{ftn:ERK1+U1-SL-equiv}), in which ERK1+U1 on the fine grid is paired with a first-order semi-Lagrangian discretization on the coarse-grid.
For this specific case the fine-grid method is equivalent to a semi-Lagrangian method, so \cite[Thm. 6.4]{DeSterck_etal_2022_LFA} can be used to show that the two-level convergence factor \eqref{eq:rho-asym} exceeds unity on the interval $c/c_{\rm max} \in \big(\tfrac{1}{m}, 1-\tfrac{1}{m} \big)$, which quickly limits to the entire range of possible CFL numbers as $m$ is increased.

\subsection{Numerical results}
\label{sec:MOL-redisc-numerical}

Now we numerically verify Theorem~\ref{thm:MOL-rho-lwr-bnd} and expand upon it by examining the convergence factor \eqref{eq:rho-omega-theta} for all modes ${\omega \in [-\pi, \pi)}$, not just those that are asymptotically smooth (as in Theorem~\ref{thm:MOL-rho-lwr-bnd}). 
%

%
\begin{figure}[t!]
\centerline{
\includegraphics[scale=0.32]{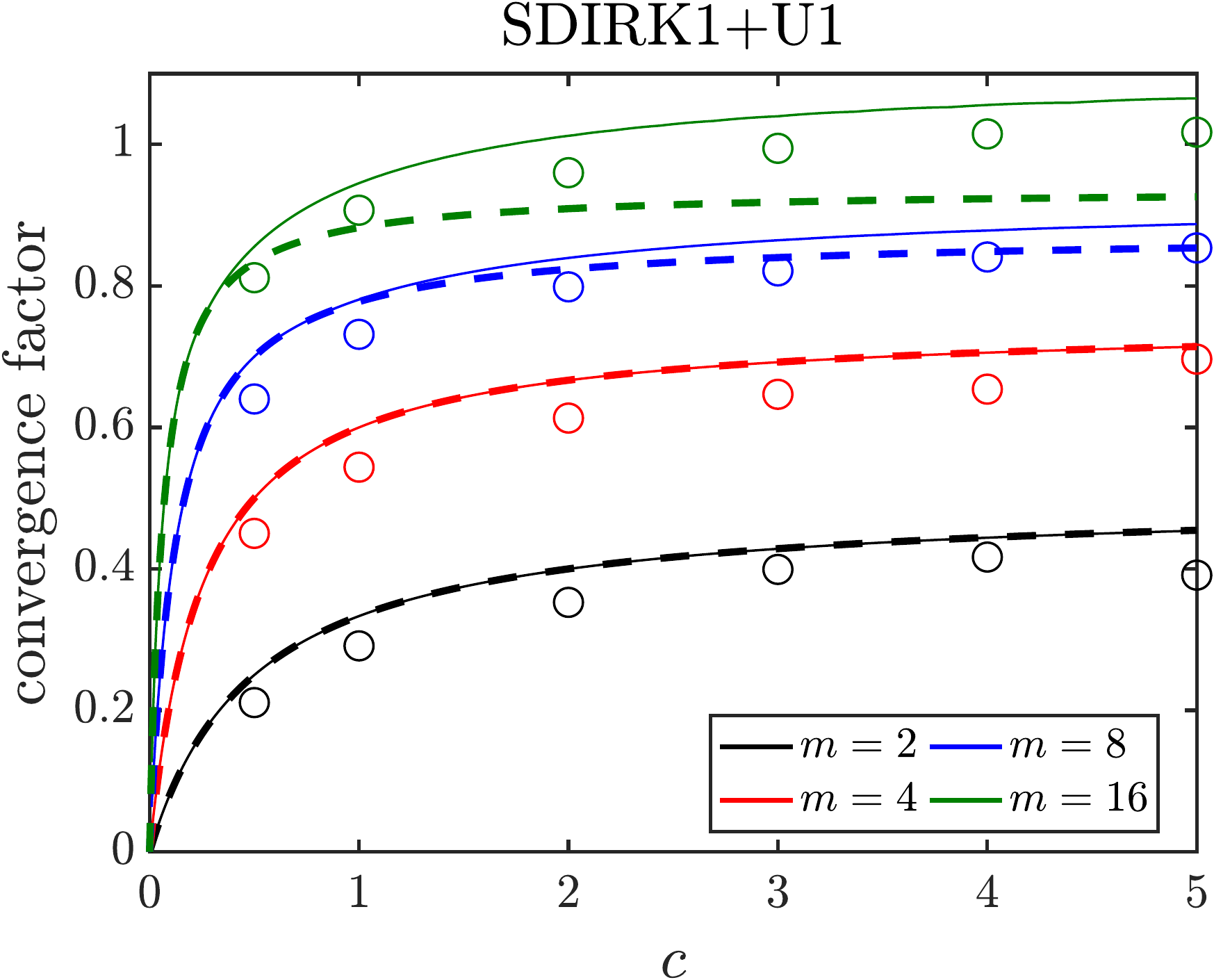}
\quad
\includegraphics[scale=0.32]{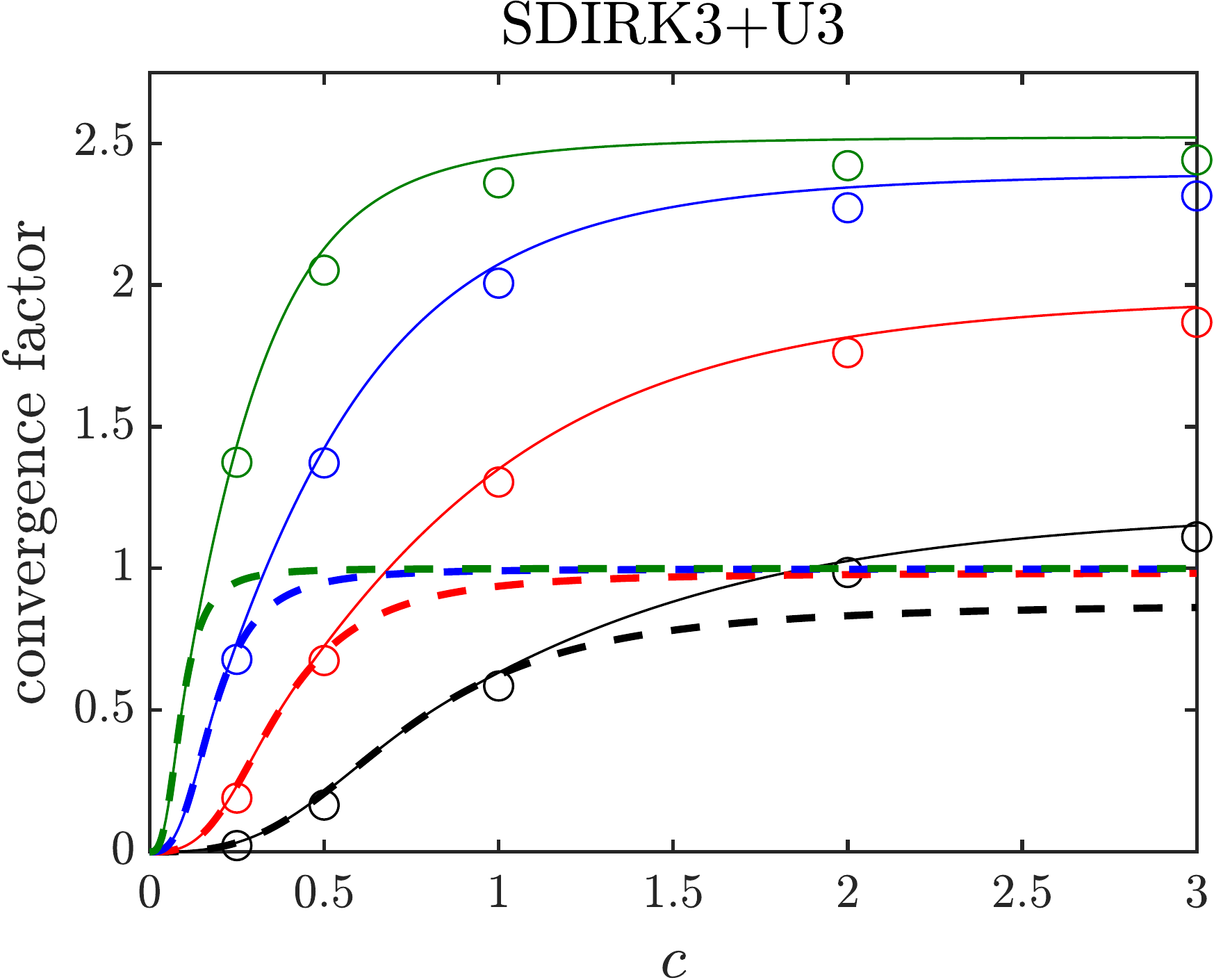}
}
\caption{
MGRIT two-grid convergence factors for implicit method-of-lines discretizations of order $p = 1$ (left), and $p = 3$ (right), as a function of the fine-grid CFL number $c$ when rediscretizing on the coarse grid. 
MGRIT uses FCF-relaxation ($\nu = 1$), and a coarsening factor of $m$.
Thin solid lines are the two-level LFA convergence factor \eqref{eq:rho-asym}.
Thick dashed lines are the function $\widecheck{\rho}_p ( c ) $ from \eqref{eq:rho-check-MOL} that acts as a lower bound on this convergence factor.
Open circle markers are experimentally determined effective convergence factors on finite-length time domains.
Note the different axis limits for each plot.
\label{fig:rho_vs_CFL-SDIRK-redisc}
}
\end{figure}

Fig.~\ref{fig:rho_vs_CFL-SDIRK-redisc} shows, as a function of fine-grid CFL number, the two-level MGRIT convergence factor \eqref{eq:rho-asym} (the thin solid lines) for two implicit discretizations. 
Clearly, MGRIT convergence is not fast for all combinations of $m$ and $c$, with the two-level convergence factor often greater than one-half for $p = 1$ and often greater than one for $p = 3$; recall that a convergence factor bigger than one indicates that MGRIT diverges. 
For fixed $m$, this convergence factor seems to converge to a constant for $c \gg 1$, with this constant being larger than unity for sufficiently large $m$ (recall from the beginning of Section~\ref{sec:MOL-redisc-analysis} that practically relevant coarsening factors satisfy $m \gg 1$).
We note that the convergence estimates for SDIRK1+U1 in Fig.~\ref{fig:rho_vs_CFL-SDIRK-redisc} are consistent with the maxima of the convergence factors reported in \cite[Fig. 4.1(c)]{Dobrev_etal_2017}, which considered the same discretization for $m = 2$ and $c \in \{1, 4\}$.

To verify that \eqref{eq:rho-asym} is an accurate predictor of the effective convergence factor on finite-length time domains, in Fig.~\ref{fig:rho_vs_CFL-SDIRK-redisc} we overlay effective convergence factors (the open circles) as measured from numerical tests.
These numerical tests use \ptxt{$n_x \times n_t = 2^{10} \times 2^{12}$} degrees-of-freedom and the convergence factor from the final iteration, with MGRIT being run either until the $\ell^2$-norm of the residual was decreased by at least 10 orders of magnitude, or until the number of iterations would exceed 20. 
\ptxt{
In these tests, the initial MGRIT iterate was taken to be uniformly random, except at $t = 0$ where it matches the initial condition in \eqref{eq:ad}, which we take as ${u_0(x) = \sin^4 (\pi x)}$. 
The MGRIT implementation is provided by XBraid \cite{xbraid}, which we have modified so that it applies the linear method described in Section \ref{sec:MGRIT} rather than an FAS version of it (see further details in Sec. 2 of the Supplementary Materials from \cite{DeSterck_etal_2022}).
}

Finally, also shown in Fig.~\ref{fig:rho_vs_CFL-SDIRK-redisc} is the lower bound \eqref{eq:rho-lwr-bnd} from Theorem~\ref{thm:MOL-rho-lwr-bnd} (the thick dashed lines).
In all cases, the lower bound appears tight for small coarse-grid CFL numbers, $m c \ll 1$; for larger $mc$, the bound is not tight in most cases.
In regions where the lower bound is tight, it means that the convergence of asymptotically smooth characteristic components is the slowest of all space-time modes.
To better understand what is occurring where the lower bound is not tight, we have considered contour plots of the MGRIT spectral radius \eqref{eq:rho-omega-theta} in Fourier space. 
We omit these plots here for brevity; see \cite[Fig. 1]{DeSterck_etal_2022_LFA} for related plots.
These plots indicate that the slowest-converging modes are no longer asymptotically smooth, but are nonetheless smooth relative to the most oscillatory modes $|\omega| \approx \pi$.

To summarize the findings in this section, the slow convergence of asymptotically smooth characteristic components precludes fast and robust MGRIT convergence when directly discretizing on the coarse grid. 
However, there may also exist other relatively smooth modes with ${\cal O}(h) < |\omega| \ll \pi$ that are not damped effectively by MGRIT. 
Since both of these sets of modes are smooth, they cannot be efficiently damped through additional relaxation, and instead must be targeted by an improved coarse-grid correction (see also \cite[Sec. 6.1]{DeSterck_etal_2022_LFA}).
%

\section{Coarse-grid operators with corrected truncation error}
\label{sec:MOL-trunc-correction}

In this section, we develop improved coarse-grid operators for method-of-lines discretizations.
The two-level coarse-grid operators are described in Section~\ref{sec:MOL-trunc-correction-operator}, with generalizations to the multilevel setting given in Section~\ref{sec:MOL-trunc-correction-operator-multilevel}, and numerical results in Sections~\ref{sec:MOL-trunc-correction-num-results-diss}~and~\ref{sec:MOL-trunc-correction-num-results-disp}.

\subsection{Coarse-grid operators \ptxt{for fine-grid method-of-lines discretizations}}
\label{sec:MOL-trunc-correction-operator}

The coarse-grid operators developed 
\ptxt{in this section are motivated by those we developed previously in \cite[Sec. 3]{DeSterck_etal_2022} (which we briefly recalled in Section \ref{sec:SL})} 
for fine-grid semi-Lagrangian discretizations, and by the characteristic component analysis from Section~\ref{sec:MOL-redisc-analysis}.
Recall from Section~\ref{sec:MOL-redisc-analysis} that we showed direct method-of-lines-based coarse-grid operators fail, at least in part, because of a mismatch between their truncation error and that of the ideal coarse-grid operator's for asymptotically smooth characteristic components.
For analogous reasons, as discussed at the end of Section~\ref{sec:MOL-redisc-theory}, a direct coarse-grid semi-Lagrangian discretization also fails to yield satisfactory MGRIT convergence when paired with a fine-grid method-of-lines discretization.
Our idea here then is to create a coarse-grid modified semi-Lagrangian operator with a truncation error that matches to lowest order the ideal coarse-grid operator for a fine-grid method-of-lines operator.
The following lemma lays the foundations for how such a modified semi-Lagrangian family of coarse-grid operators may be constructed.
\begin{lemma}[Ideal coarse-grid operator perturbation] \label{lem:Psi_ideal-MOL-pert}
Suppose $\Phi = {\cal M}_{p,p}^{(\delta t)}$, with spatial and temporal error constants $\eFD$ and $\eRK$, respectively.
Suppose the solution of PDE \eqref{eq:ad} is sufficiently smooth, and let $\bm{u}(t_{n}) \in \mathbb{R}^{n_x}$ denote this solution at time $t_n$ sampled at the spatial mesh points.
Then, the action of the ideal coarse-grid operator $\Psi_{\rm ideal} = \Phi^m$ can be expressed as a perturbation to a direct, coarse-grid semi-Lagrangian discretization ${\cal S}_p^{(m \delta t)}$ in the following two ways
\begin{align}
\label{eq:Psi_ideal-SL_redisc-pert0}
\Psi_{\rm ideal} \bm{u}(t_{n}) 
&=
\Big(
I
+ 
\varphi^{(m \delta t)}_{p+1} {\cal D}^{(p+1)}_s 
\Big) 
{\cal S}_p^{(m \delta t)} \bm{u}(t_{n}) 
+ 
{\cal O}\big( 
h^{p+2}, 
h^{p+1} \delta t,
h \delta t^{p+1}, 
\delta t^{p+2} \big),
\\
\label{eq:Psi_ideal-SL_redisc-pert}
&
= 
\Big( I - \varphi^{(m \delta t)}_{p+1} {\cal D}^{(p+1)}_s \Big)^{-1} {\cal S}_p^{(m \delta t)} \bm{u}(t_n) 
+
{\cal O}\big( 
h^{p+2}, 
h^{p+1} \delta t,
h \delta t^{p+1}, 
\delta t^{p+2} \big),
\end{align}
where the constant $\varphi^{(m \delta t)}_{p+1}$ is given by
\begin{align} \label{eq:varphi-constant-MOL-SL}
\varphi^{(m \delta t)}_{p+1} 
= 
m
\left[
c  
\mkern 1mu 
\eFD 
+
(-c)^{p+1} \eRK 
\right]
+
(-1)^{p+1} f_{p+1} \big( \varepsilon^{(m \delta t)} \big).
\end{align}
In \eqref{eq:varphi-constant-MOL-SL}, the term $(-1)^{p+1} f_{p+1} \big( \varepsilon^{(m \delta t)} \big)$ is associated with the truncation error of the coarse-grid semi-Lagrangian discretization, see Lemma~\ref{lem:SL_trunc}.
\end{lemma}
\begin{proof}
See Appendix~\ref{app:Psi_ideal_MOL-pert-proof}.
\qed
\end{proof}

The second term on the right-hand side of \eqref{eq:Psi_ideal-SL_redisc-pert0} is of size ${\cal O}(h^{p+1})$, since ${\cal D}^{(p+1)}_s$ and ${\cal S}_p^{(m \delta t)}$ commute, \ptxt{and ${\cal D}^{(p+1)}_s\bm{u}(t_{n}) = {\cal O}(h^{p+1})$.}
Thus, applied to $\bm{u}(t_n)$, the ideal coarse-grid operator and the semi-Lagrangian discretization are consistent up to terms of size ${\cal O}(h^{p+1})$, which makes sense given they are both $p$th-order accurate discretizations of the same PDE.
As stated above, however, this consistency alone does not produce satisfactory MGRIT convergence when using $\Psi = {\cal S}_p^{(m \delta t)}$.
Instead, improved MGRIT convergence necessitates that we correct the semi-Lagrangian discretization to increase its accuracy with respect to $\Psi_{\rm ideal}$.
Thus, based on the estimate \eqref{eq:Psi_ideal-SL_redisc-pert}, we propose the following coarse-grid operator for this purpose:
\begin{align} \label{eq:MOL-SL-corrected}
\Psi = \Big( I - \varphi^{(m \delta t)}_{p+1} {\cal D}^{(p+1)}_s \Big)^{-1} {\cal S}_p^{(m \delta t)}.
\end{align}

As mentioned in Footnote~\ref{ftn:ERK1+U1-SL-equiv}, for the constant-wave-speed advection problem, ERK1+U1 is equivalent to ${\cal S}_1^{(\delta t)}$ if $c < 1$. 
Thus, when $\Phi = $ ERK1+U1 and $c < 1$, the coarse-grid operator \eqref{eq:MOL-SL-corrected} is equivalent to that of \eqref{eq:SL-redisc-corrected} when $\Phi = {\cal S}_1^{(\delta t)}$ and $c < 1$ (i.e., although the formulae for the correction coefficient in \eqref{eq:MOL-SL-corrected} and \eqref{eq:SL-redisc-corrected} are different in general, for this special case they produce the same values). 

Applying the coarse-grid operator \eqref{eq:MOL-SL-corrected} requires solving a linear system, \ptxt{potentially} making it substantially more expensive than a standalone semi-Lagrangian discretization \ptxt{(depending on how accurately the linear system is solved)}. 
\ptxt{In the forthcoming numerical tests, depending on the test, this system is solved directly with UMFPACK \cite{Davis2004} from SuiteSparse, or it is solved efficiently and approximately with a small number of GMRES iterations (see also \cite{DeSterck_etal_2022}, \cite[Chap. 4]{KrzysikThesis2021} for more details on this in the context of fine-grid semi-Lagrangian discretizations).
}

\begin{remark}[Explicit versus implicit-explicit coarse-grid operator]
From \eqref{eq:Psi_ideal-SL_redisc-pert0}, notice that a more straightforward, and less expensive, coarse-grid operator than \eqref{eq:MOL-SL-corrected} which still applies the truncation error correction is 
$\Psi = \big( I + \varphi^{(m \delta t)}_{p+1} {\cal D}^{(p+1)}_s \big) {\cal S}_p^{(m \delta t)}$. 
However, this coarse-grid operator is unstable (i.e., $\Vert \Psi \Vert_2 > 1$) for sufficiently large $m$, and thus cannot serve as a suitable coarse-grid operator.
(The specific value of $m$ for which instability arises depends on problems parameters such as the CFL number, discretization order, etc. For example, \ptxt{in the case where a fine-grid semi-Lagrangian discretization is used,} it can be as small as $m = 4$ in some cases, see \cite[Sec. 4.2.4]{KrzysikThesis2021}.)
\end{remark}

\begin{remark}[A perturbed, implicit method-of-lines coarse-grid operator]
Considering the coarse-grid operator \eqref{eq:MOL-SL-corrected}, a natural question is why we elect to base it on a coarse-grid semi-Lagrangian discretization rather than an implicit method-of-lines discretization.
That is, rather than \eqref{eq:MOL-SL-corrected}, why not use something like \\ 
${
\big( I - \varphi^{(m \delta t)}_{p+1} {\cal D}^{(p+1)}_s \big)^{-1} {\cal M}_{p,p}^{(m \delta t)}
},
$
with $\varphi^{(m \delta t)}_{p+1}$ not given by \eqref{eq:varphi-constant-MOL-SL}, but instead chosen so as to take into account the truncation error of ${\cal M}_{p,p}^{(m \delta t)}$ rather than ${\cal S}_{p}^{(m \delta t)}$?
Our numerical tests indicate that such a coarse-grid operator yields MGRIT convergence that is much less satisfactory than that of \eqref{eq:MOL-SL-corrected}, often leading to divergence.
\end{remark}

\subsection{Coarse-grid operators: The multilevel setting}
\label{sec:MOL-trunc-correction-operator-multilevel}

Following \ptxt{the steps used in \cite[Sec. 3.4]{DeSterck_etal_2022} to develop coarse operators in the multilevel setting}, we generalize the two-level coarse-grid operator \eqref{eq:MOL-SL-corrected} to any coarse level $\ell > 0$ in a multilevel hierarchy.
Specifically, suppose we have a multilevel hierarchy, with level $\ell \geq 0$ employing a time-step size of $m^{\ell} \delta t$.
Let $\Phi^{(m^{\ell} \delta t)} \in \mathbb{R}^{n_x \times n_x}$ be a time-stepping operator that evolves solutions forwards in time by an amount $m^{\ell} \delta t$.\footnote{Observe the change in notation here; previously we have denoted level $\ell = 0$ and $\ell = 1$ time-stepping operators by $\Phi$ and $\Psi$, respectively. That is, in terms of our earlier notation $\Phi \equiv \Phi^{(\delta t)}$ and $\Psi \equiv \Phi^{(m \delta t)}$.}
Then, we propose the coarse-grid operator
\begin{align} \label{eq:MOL-SL-corrected-multilevel}
\Phi^{(m^{\ell} \delta t)} 
= 
\Big( 
I - \varphi^{(m^{\ell} \delta t)}_{p+1} {\cal D}^{(p+1)}_s 
\Big)^{-1} {\cal S}_p^{(m^{\ell} \delta t)}, 
\quad
\ell = 1, 2, \ldots,
\end{align}
with level-dependent correction coefficient $\varphi^{(m^{\ell} \delta t)}_{p+1}$ defined recursively by
\begin{align}
\varphi^{(m^{\ell} \delta t)}_{p+1}
=
\begin{cases}
m
\left[
c  
\mkern 1mu 
\eFD 
+
(-c)^{p+1} \eRK 
\right]
+
(-1)^{p+1} f_{p+1} \big( \varepsilon^{(m \delta t)} \big), 
\quad &\ell = 1,
\\[2ex]
(-1)^{p+1} 
\Big[ 
- 
m f_{p+1} \big( \varepsilon^{(m^{\ell - 1} \delta t)} \big) 
+
f_{p+1} \big( \varepsilon^{(m^{\ell} \delta t)} \big) 
\Big]
+
m \varphi^{(m^{\ell - 1} \delta t)}_{p+1}, 
\quad 
&\ell > 1.
\end{cases}
\end{align}

\subsection{Numerical results: Dissipative discretizations (odd $p$)}
\label{sec:MOL-trunc-correction-num-results-diss}

We now examine numerically the effectiveness of the coarse-grid operators developed in Sections~\ref{sec:MOL-trunc-correction-operator}~and~\ref{sec:MOL-trunc-correction-operator-multilevel}.
Results for dissipative discretizations ($p$ odd), in the sense of Definition~\ref{def:diss_vs_disp}, are given in this section, while results for dispersive discretizations ($p$ even) are given next in Section~\ref{sec:MOL-trunc-correction-num-results-disp}.

The numerical tests in this section take the differentiation matrix ${\cal D}^{(p+1)}_s$ (see \eqref{eq:Dp+1_def}) in the coarse-grid operators \eqref{eq:MOL-SL-corrected} and \eqref{eq:MOL-SL-corrected-multilevel} as a symmetric, second-order accurate discretization \ptxt{(i.e., $s = 2$)}.
\ptxt{
For SDIRK tests, all coarse-grid linear systems arising from the application of \eqref{eq:MOL-SL-corrected} and \eqref{eq:MOL-SL-corrected-multilevel} are solved directly because the cost of doing so is not large relative to the direct linear solves carried out during time-stepping on the finest grid. 
For two-level ERK tests, the linear systems are solved directly so as to clearly understand best-case MGRIT convergence, while for multilevel tests they are solved approximately on all levels with GMRES.
GMRES is iterated until the $\ell^2$-norm of the relative residual has decreased by at least two orders of magnitude, or until a maximum number of iterations is reached (10 for ERK1+U1, and 20 for ERK3+U3 and ERK5+U5). These GMRES halting criteria are chosen so that MGRIT convergence does not deteriorate relative to when a direct solver is used; however, they have not been fully optimized to minimize the number of GMRES iterations.}

\ptxt{Numerical results are split across several sections. First we focus on MGRIT convergence (Section \ref{sec:dissipative-convergence}), followed by parallel scaling studies (Section \ref{sec:parallel}), and then more detailed investigations of some interesting convergence results for specific cases (Sections \ref{sec:CFL-limit-deterioration} and \ref{sec:F-vs-FCF}).}
%

\subsubsection{\ptxt{MGRIT convergence study}}
\label{sec:dissipative-convergence}

\begin{figure}[b!]
\centerline{
\includegraphics[scale=0.32]{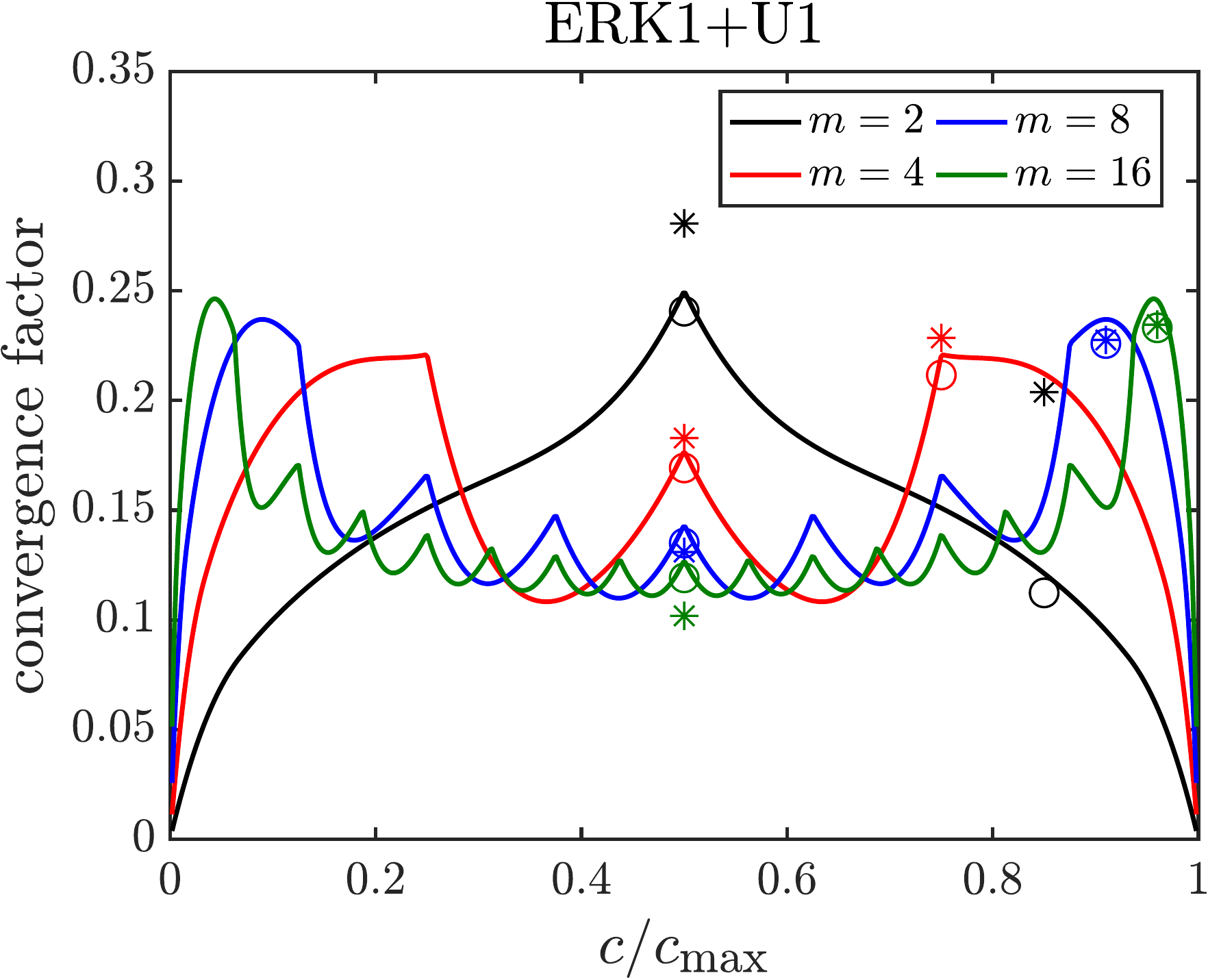}
\quad
\includegraphics[scale=0.32]{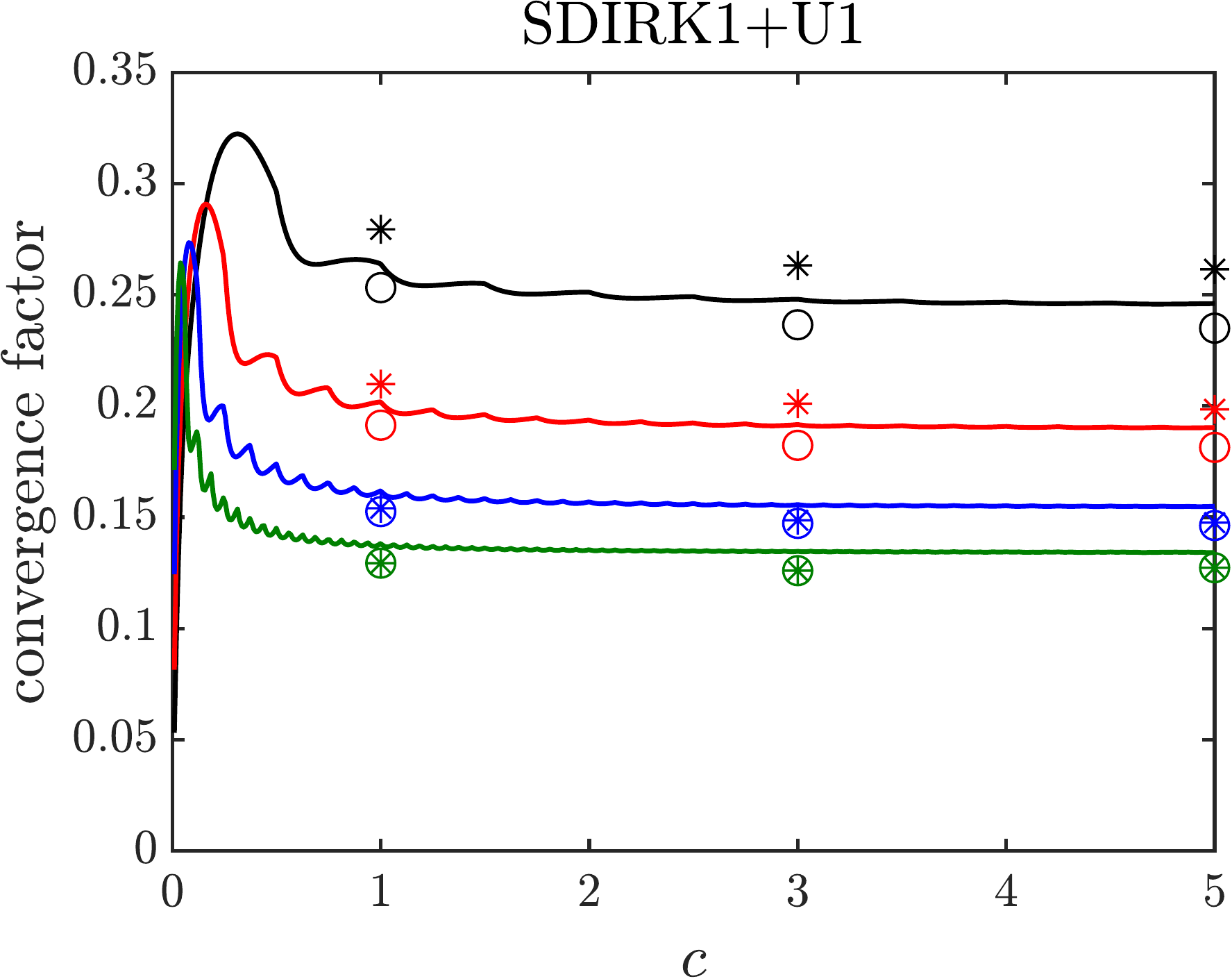}
}
\vspace{2ex}
\centerline{
\includegraphics[scale=0.32]{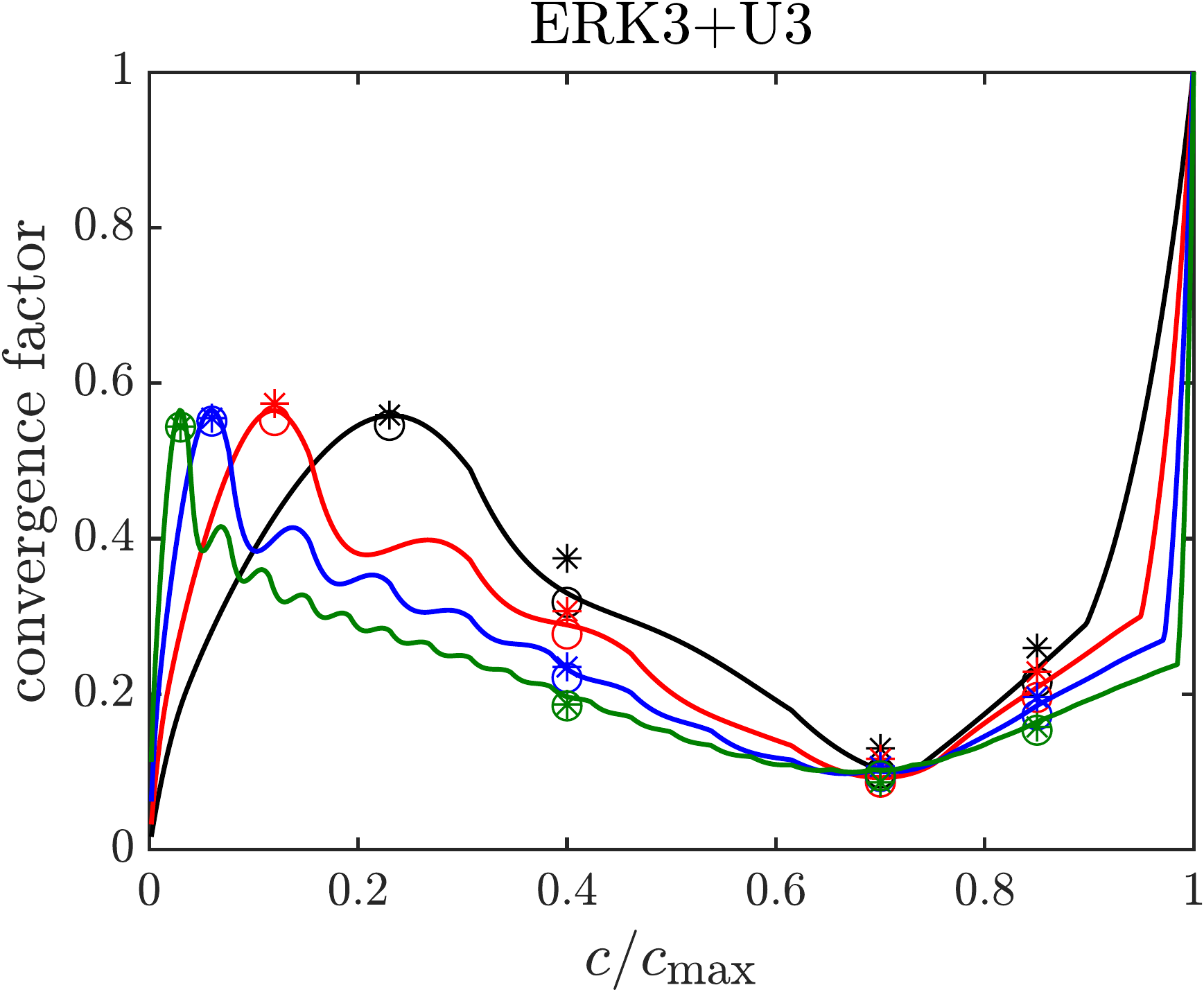}
\quad
\includegraphics[scale=0.32]{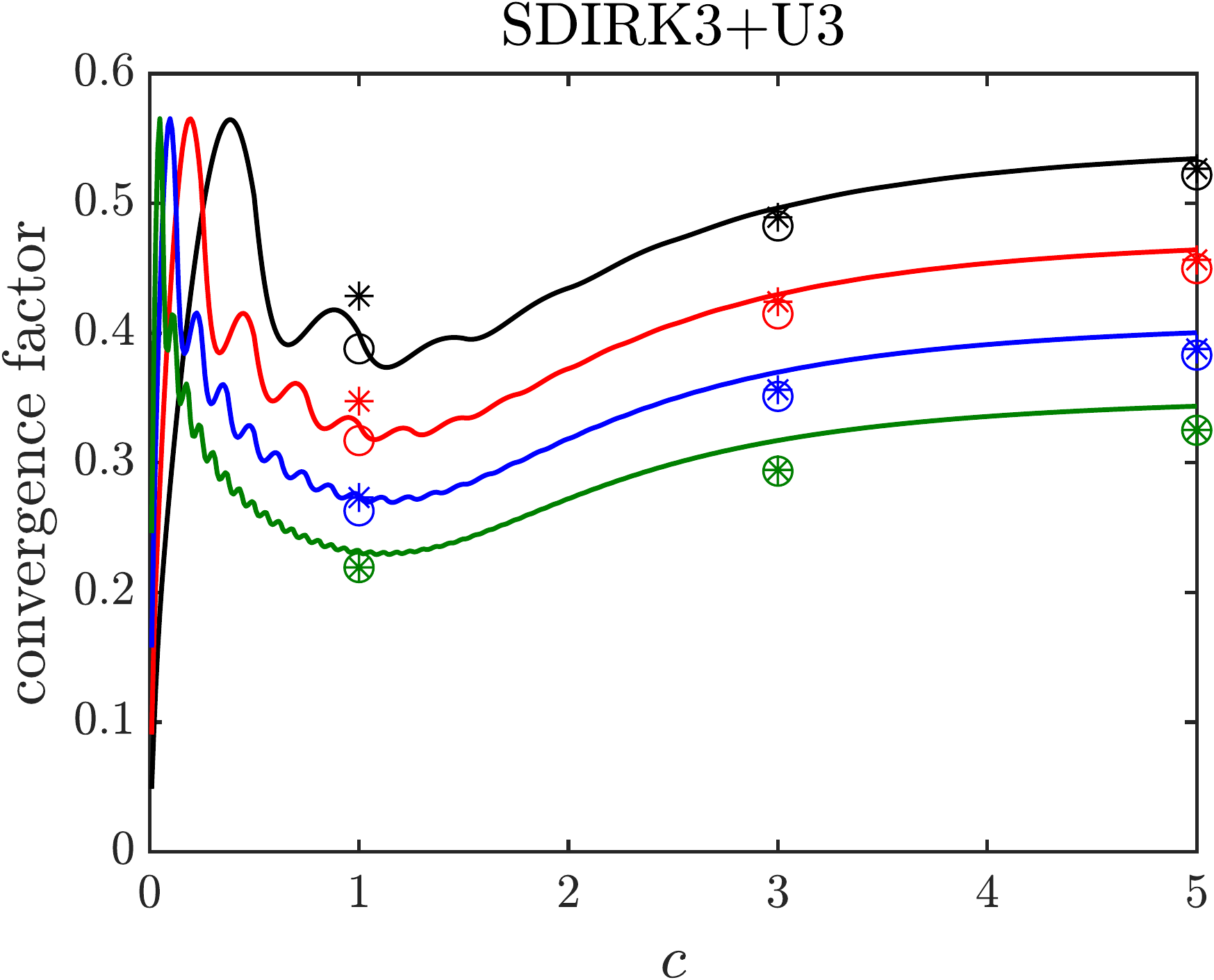}
}
\vspace{2ex}
\centerline{
\includegraphics[scale=0.32]{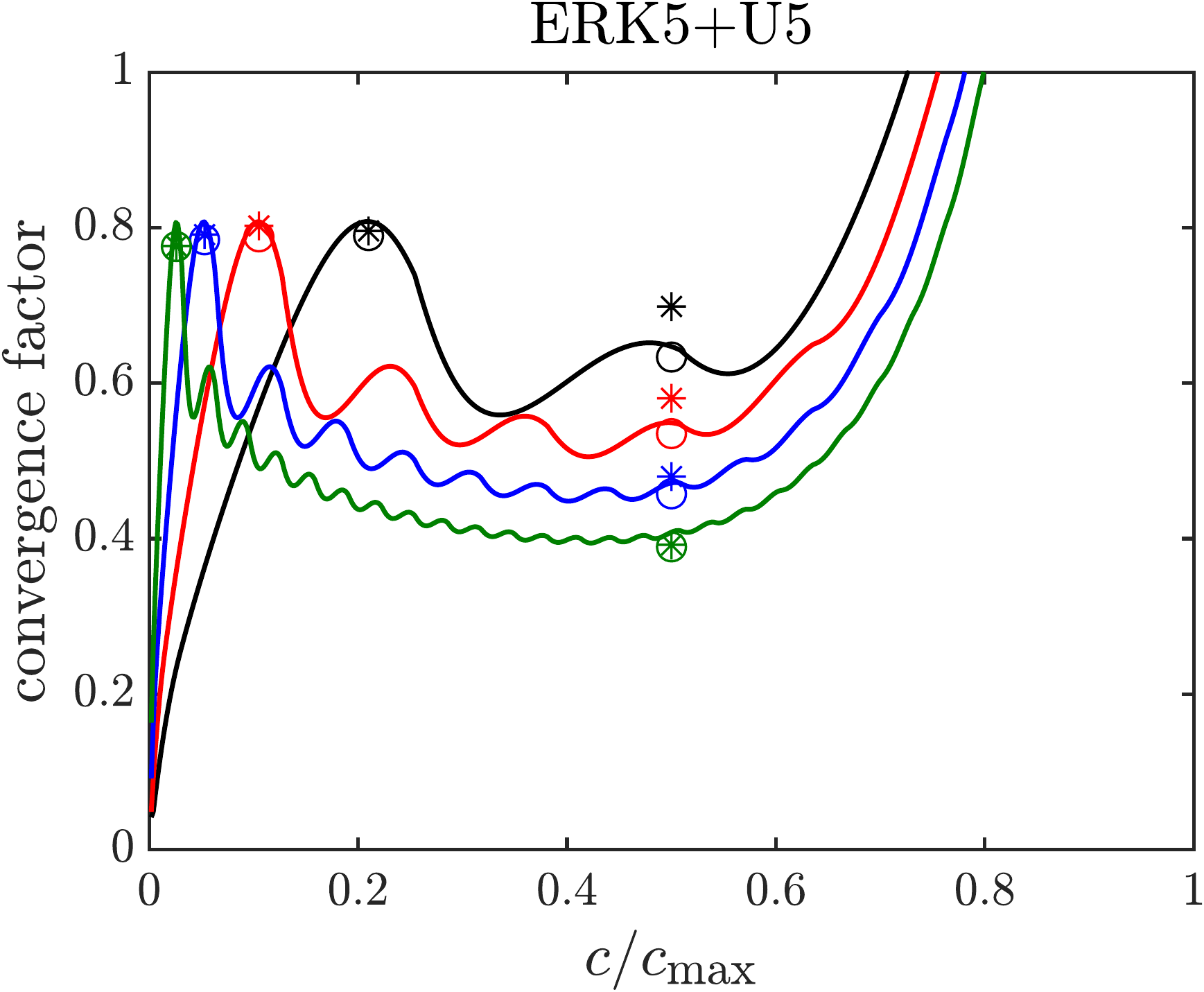}
\quad
\includegraphics[scale=0.32]{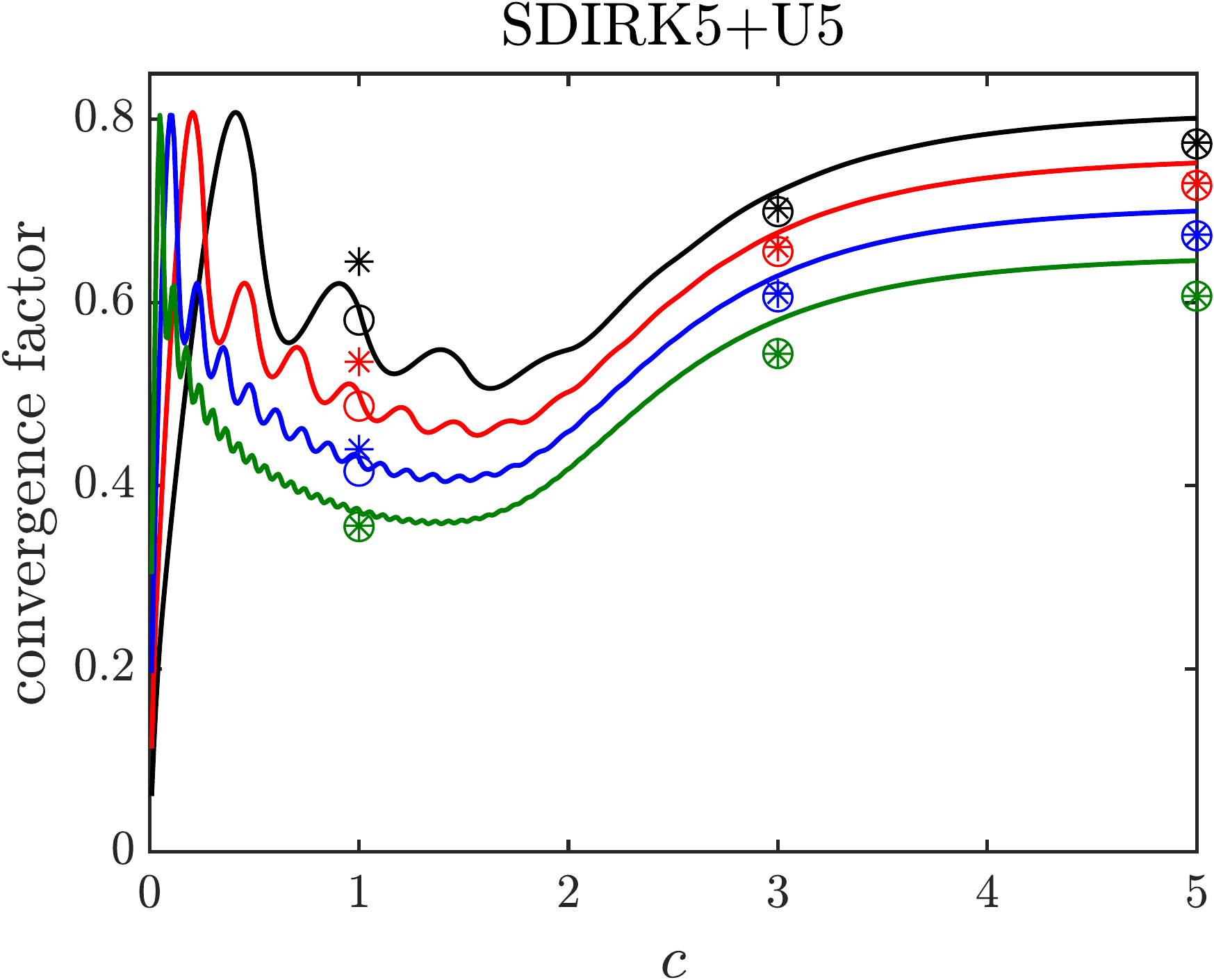}
}
\caption{
Dissipative discretizations:
Convergence factors for method-of-lines discretizations as a function of the fine-grid CFL number $c$ when using the corrected semi-Lagrangian operator \eqref{eq:MOL-SL-corrected} on coarse levels. 
Left: Explicit discretizations. 
Right: Implicit discretizations.
MGRIT uses FCF-relaxation ($\nu = 1$), and a coarsening factor of $m$.
Solid lines are the two-level LFA convergence factor \eqref{eq:rho-asym}.
Markers are effective MGRIT convergence factors on the finite interval $t \in (0, T]$, with circles representing two-level measurements, and asterisks representing multilevel measurements.
Note the different axis scales across the plots.
\label{fig:rho_vs_CFL-MOL-SL-correct-diss}
}
\end{figure}

\ptxt{We begin by assessing the effectiveness of the coarse-grid operators by examining the two-level MGRIT convergence factor \eqref{eq:rho-asym} as a function of the fine-grid CFL number (as we did in Section~\ref{sec:MOL-redisc-numerical} for direct coarse operators).}
Plots for explicit and implicit discretizations of orders one, three, and five are given in Fig.~\ref{fig:rho_vs_CFL-MOL-SL-correct-diss}.
Notice that convergence plots for explicit discretizations use the CFL fraction rather than CFL number (Table~\ref{tab:CFL_limits}).

Markers overlaid on the plots in Fig.~\ref{fig:rho_vs_CFL-MOL-SL-correct-diss} are effective MGRIT convergence factors, as measured by numerical experiments. \ptxt{These experiments use the same MGRIT setup as described in Section \ref{sec:MOL-redisc-numerical}, except that the maximum number of iterations is increased from 20 to 30.}
Circle markers are from two-level tests, \ptxt{and asterisks from V-cycle tests,  which coarsen until fewer than two points in time would result.}

In stark contrast to the results from Section~\ref{sec:MOL-redisc-numerical} for direct coarse-grid operators (see Fig. \ref{fig:rho_vs_CFL-SDIRK-redisc}), the modified coarse-grid operators result in convergent MGRIT solvers for almost all combinations of coarsening factor and CFL number pictured.
Moreover, the convergence rates pictured are fast in many cases relative to convergence rates that are typically achievable for the two-grid solution of hyperbolic problems.
Analogously to our previous work in \cite{DeSterck_etal_2022}, when a semi-Lagrangian discretization is used on the fine grid, there is also some overall degradation in the convergence factor here as the order of the discretization is increased.
%

%
%
\renewcommand*{\arraystretch}{1.3} 
\begin{table}[b!]
\caption{
	Number of two-level MGRIT iterations to reduce the $\ell^2$-norm of the residual by 10 orders of magnitude, with corresponding V-cycle iteration counts shown in parentheses. 
	MGRIT uses FCF-relaxation, and the modified coarse-grid operator \eqref{eq:MOL-SL-corrected-multilevel}.
	The ERK3+U3 tests use $c = 0.85 c_{\rm max}$, and the SDIRK3+U3 tests use $c = 5$.
  	\label{tab:MGRIT_iters}
	}
  \centering
  \begin{tabular}{| c |  c | c | c | c | c |}  
\cline{2-6}
\multicolumn{1}{ c|}{}  & $n_x \times n_t$ & $m = 2$ & $m = 4$ & $m = 8$ & $m = 16$ \\\hline
\Xhline{2\arrayrulewidth}
{\multirow{3}{*}{ERK3+U3}}
 & $2^{6} \times 2^{8}$    & 13 (14)  & 12 (13) & 11 (11) & 9 (9) \\\cline{2-6}
 & $2^{8} \times 2^{10}$  & 14 (15)  & 13 (14) & 12 (13) & 11 (12) \\\cline{2-6}
 & $2^{10} \times 2^{12}$  & 14 (15)  & 14 (14) & 12 (13) & 12 (12) \\\hline
\Xhline{2\arrayrulewidth}
{\multirow{3}{*}{SDIRK3+U3}}
 & $2^{6} \times 2^{8}$    & 28 (28) & 21 (21) & 15 (15) & 9 (9) \\\cline{2-6}
 & $2^{8} \times 2^{10}$  & 29 (29) & 25 (25) & 20 (20) & 16 (16) \\\cline{2-6}
 & $2^{10} \times 2^{12}$  & 30 (30) & 25 (25) & 21 (21) & 18 (18) \\\hline
 \end{tabular}
\end{table}

To help contextualize the convergence factors in Fig.~\ref{fig:rho_vs_CFL-MOL-SL-correct-diss}, in Table~\ref{tab:MGRIT_iters} we provide the number of MGRIT iterations to reach convergence for several example problems, and also investigate scalability as a function of problem size.
Iteration counts are given as a function of total space-time degrees-of-freedom, which illustrate that the solver is scalable in that the number of iterations is (approximately) constant as the mesh is refined.

Quite remarkably, the numerical results in Fig.~\ref{fig:rho_vs_CFL-MOL-SL-correct-diss} and Table~\ref{tab:MGRIT_iters} show little, if any, deterioration in the convergence rate between two-level and multilevel solves.
Scalable V-cycle convergence for MGRIT applied to hyperbolic problems has not been achieved except in our recent works \cite{DeSterck_etal_2021,DeSterck_etal_2022}, \ptxt{with \cite{DeSterck_etal_2021} using an approach that is not practical for real problems, and \cite{DeSterck_etal_2022} only applying to (fine-grid) semi-Lagrangian schemes.}
For example, both hyperbolic method-of-lines MGRIT studies \cite{Howse_etal_2019,Hessenthaler_etal_2020} did not obtain scalable convergence even when using F-cycles, which are more expensive than V-cycles. 
%
%

\subsubsection{\ptxt{Parallel results}}
\label{sec:parallel}

\ptxt{
We now consider the parallel performance of the proposed coarse-grid operators, with strong-scaling results given in Fig. \ref{fig:strong-scaling} for a number of problems. 
Note that we parallelize only in time as we are primarily interested in temporal scalability. 
Problem sizes and the coarsening strategy are chosen to be the same as those used for the parallel tests in \cite{DeSterck_etal_2022} (see the legend of Fig. \ref{fig:strong-scaling} for specifics).
The results were generated on Ruby, a Linux cluster at Lawrence Livermore National Laboratory consisting of 1,480 compute nodes, with 56 Intel Xeon CLX-8276L cores per node.

Fig. \ref{fig:strong-scaling} shows that speed-ups relative to sequential time-stepping are obtained for all four problems considered. Speed-ups for the implicit methods are noticeably larger than for the explicit methods; this difference occurs because the cost of the coarse-grid operator relative to the cost of the finest-grid operator is less for the implicit methods than the explicit methods, and, thus, they have greater potential for speed-ups. 
Specifically for the explicit results, the solution of the coarse-grid linear systems is a significant cost, particularly on deeper levels in the multigrid hierarchy, where the required number of GMRES iterations increases. 
We anticipate that if the GMRES halting criteria were optimized, or if a more specialized solver was used, then the speed-ups would increase further. This is something we hope to address in future work. 
Note also that the halting criterion of reducing the residual by 10 orders of magnitude oversolves these problems with respect to discretization accuracy, meaning that larger speed-ups would result from tuning the halting tolerance.
%
}

\begin{figure}[b!]
\centerline{
\includegraphics[scale=0.31]{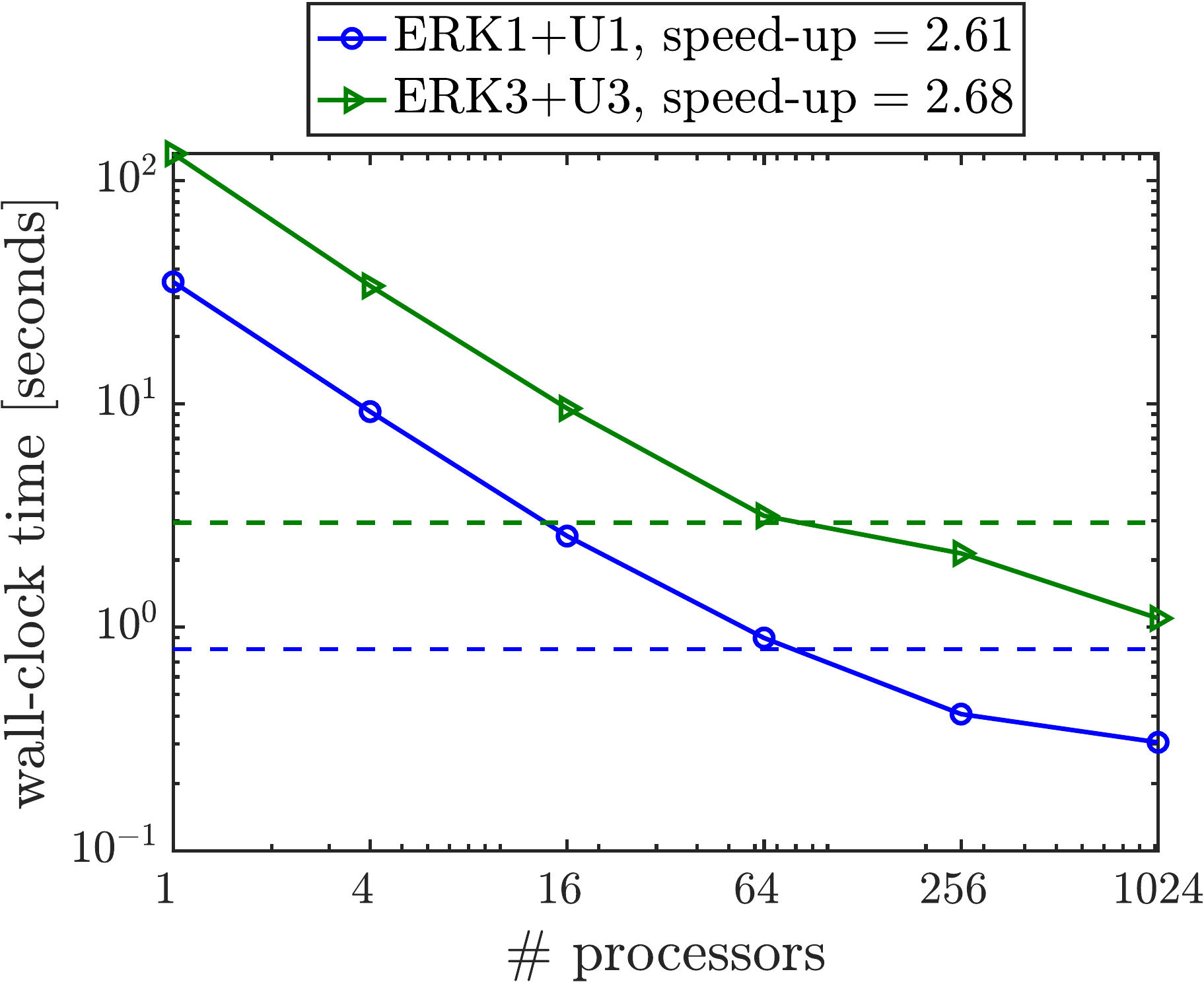}
\quad
\includegraphics[scale=0.31]{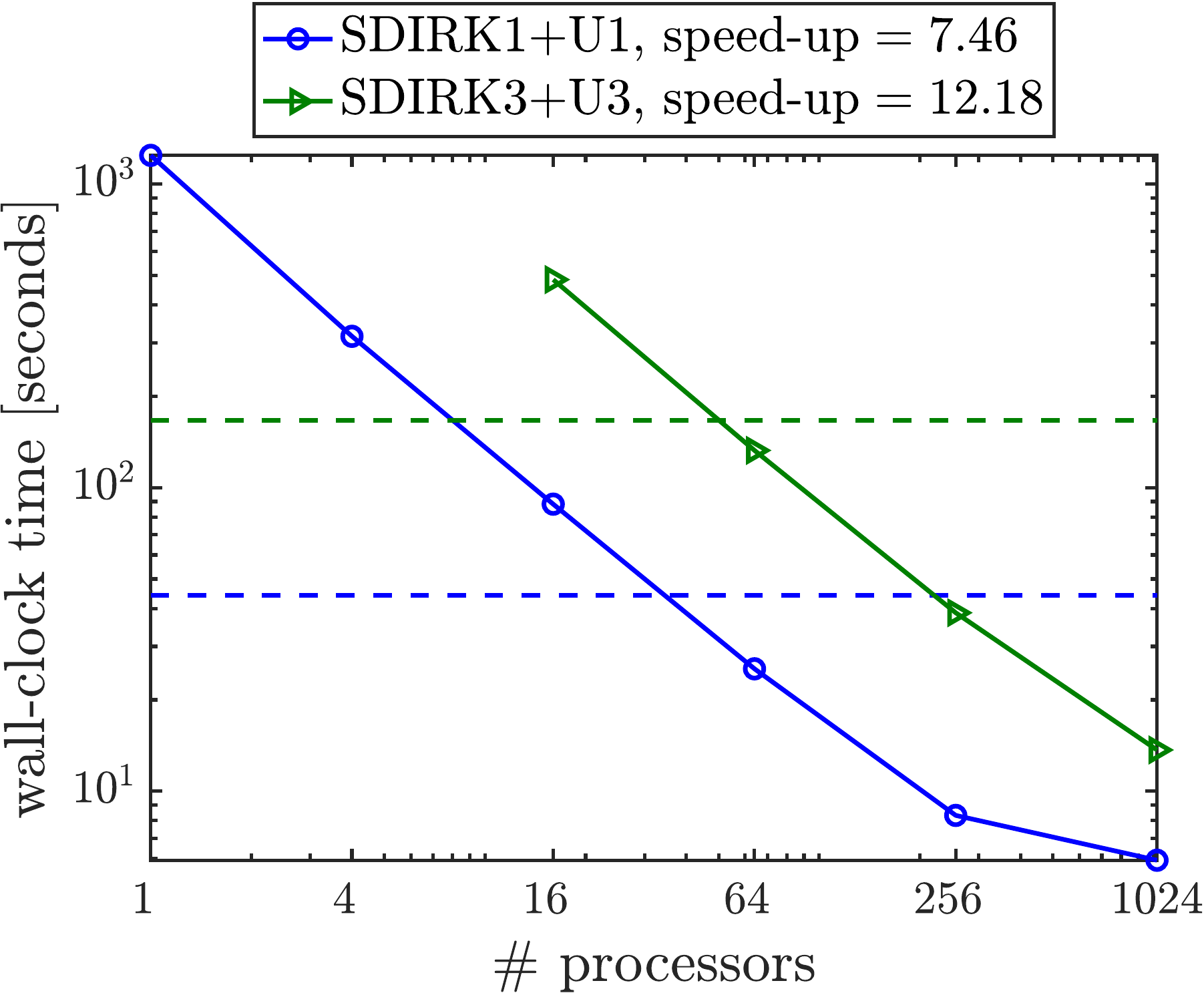}
}
\caption{
\ptxt{
Strong-scaling studies for MGRIT V-cycles applied to two explicit discretizations (left panel), and two implicit discretization (right panel).
Wall-clock times for sequential time-stepping on a single processor are given by the dashed lines; speed-ups over sequential time-stepping when using 1024 processes are shown in the legends.
MGRIT is iterated until the $\ell^2$-norm of the residual has been reduced by 10 orders of magnitude; the coarsening factor is $m = 16$ on the first level, and $m = 4$ on all coarser levels.
All problems use a space-time mesh with $n_x \times n_t = 2^{12} \times 2^{14}$ degrees-of-freedom.
The ERK1+U1 and ERK3+U3 tests use CFL numbers that are 85\% of their respective maxima, while the SDIRK1+U1 and SDIRK3+U3 tests use CFL numbers of one and five, respectively.
}
\label{fig:strong-scaling}
}
\end{figure}
%

\subsubsection{Slow convergence near the CFL limit for higher-order discretizations}
\label{sec:CFL-limit-deterioration}

For ERK3+U3 and ERK5+U5 (middle and bottom panels in the left column of Fig.~\ref{fig:rho_vs_CFL-MOL-SL-correct-diss}), MGRIT convergence degrades substantially as the fine-grid CFL number approaches its limit.
On one hand, this deterioration is not really a limitation because it can be avoided through using the discretization at smaller CFL numbers; on the other hand, for the fifth-order scheme this is somewhat restrictive since the deterioration sets in at around 60\% of the CFL limit. 

The reason for this degradation is that, as the CFL limit is approached for these discretizations, a set of oscillatory spatial modes emerge, with eigenvalues that approach one in magnitude.
That is, the CFL limit is determined by oscillatory modes that are otherwise relatively dissipative at smaller CFL numbers (i.e., their eigenvalues are not close to one in magnitude), as pictured in the left panel of Fig.~\ref{fig:CFL-limit-issue} for ERK3+U3.
It is the case that eigenvalues of the ideal coarse-grid operator with magnitude close to one need to be very accurately approximated by the coarse-grid operator, since relaxation cannot be used to damp the associated errors \cite{DeSterck_etal_2021,DeSterck_etal_2022_LFA}.
The issue now, however, is that our coarse-grid operator is unable to accurately capture these oscillatory modes that become important for MGRIT convergence as the CFL limit is approached, recalling that it is designed to mimic the ideal coarse-grid operator only for asymptotically smooth modes, $\omega \approx 0$.
The impact that the CFL limit has on MGRIT convergence in this scenario can be seen clearly in the right panel of Fig.~\ref{fig:CFL-limit-issue} for ERK3+U3 when $m = 4$.
%

%
\begin{figure}[b!]
\centerline{
\includegraphics[scale=0.33]{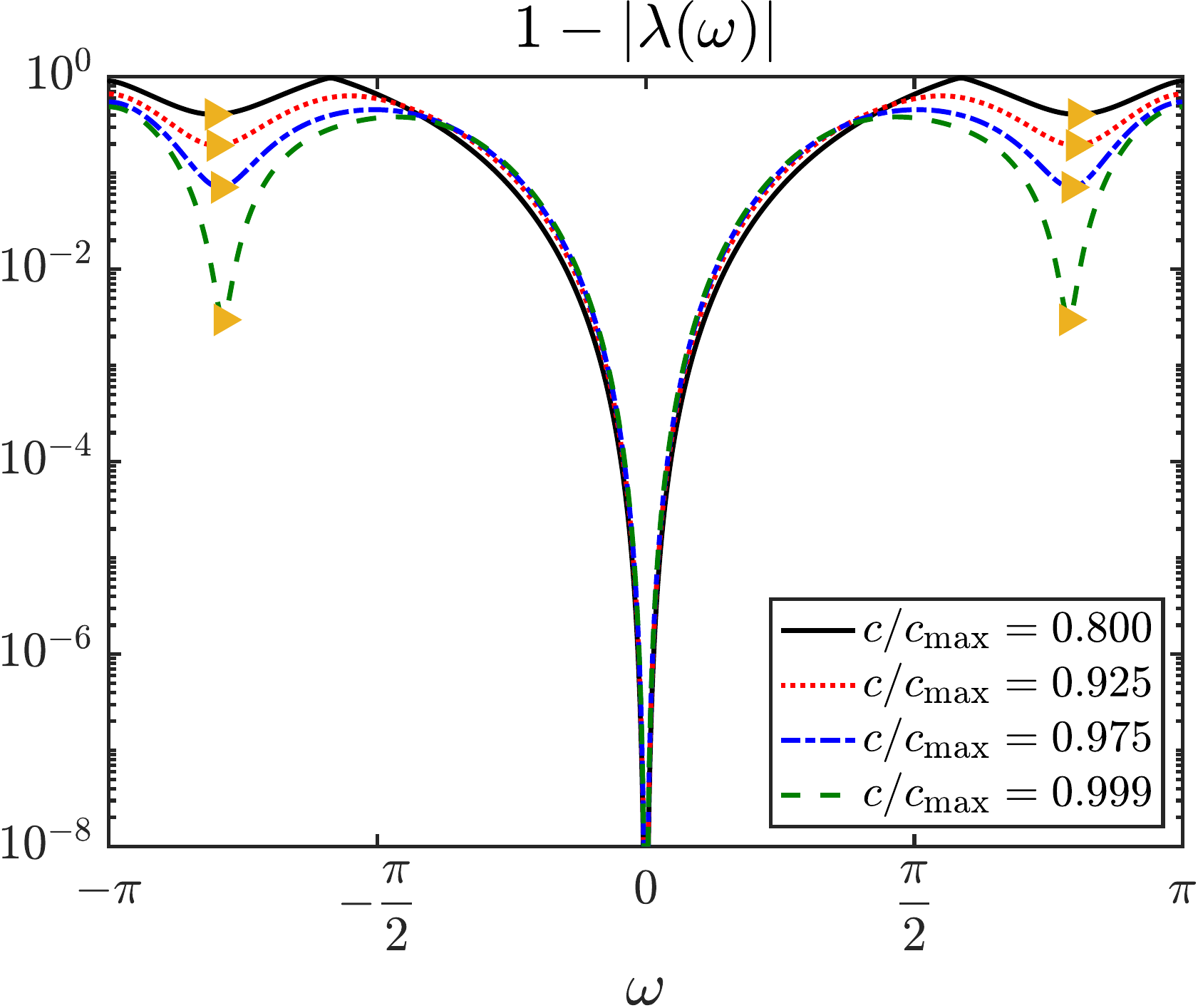}
\quad
\includegraphics[scale=0.33]{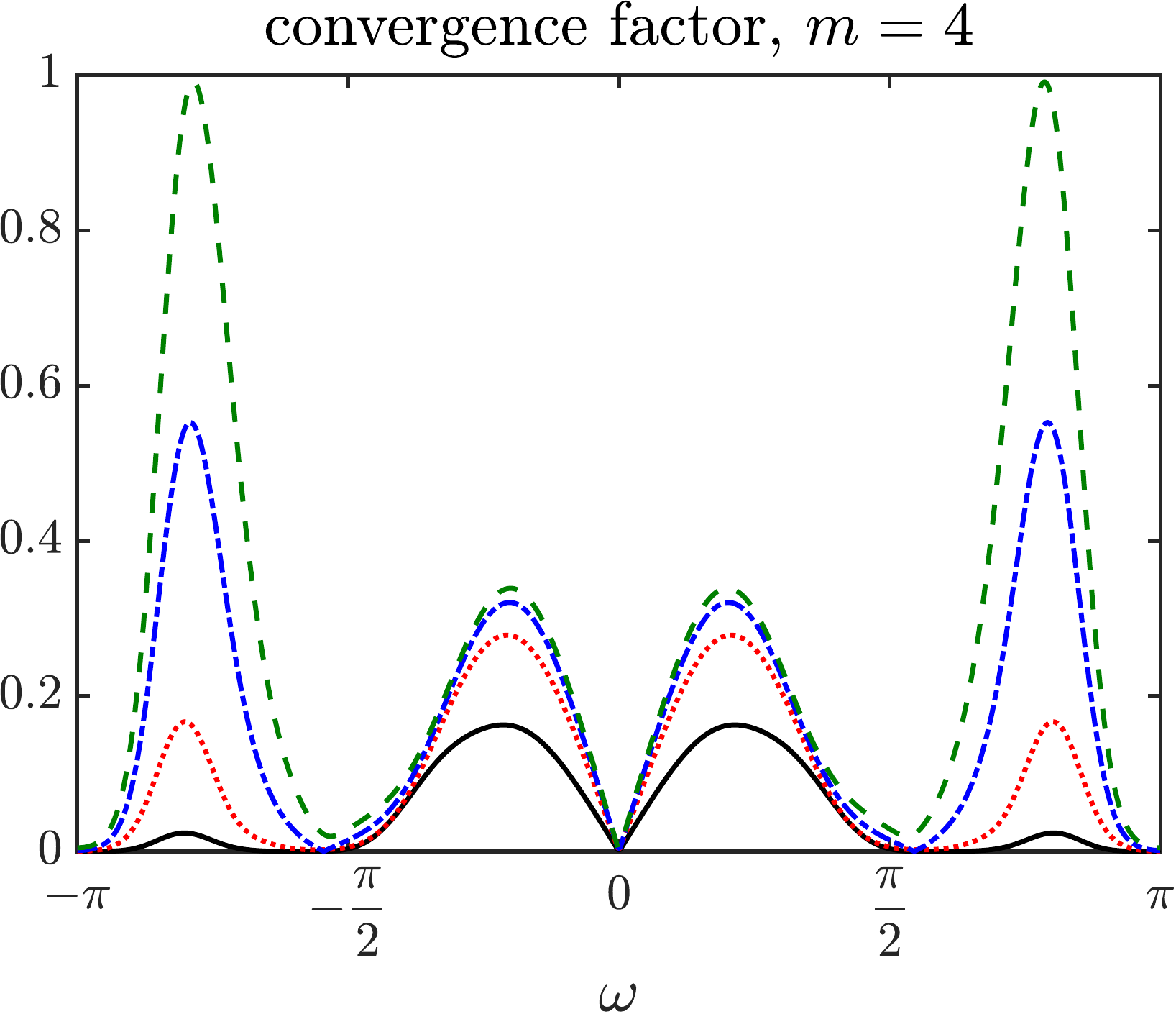}
}
\caption{
MGRIT convergence degradation for ERK3+U3 near the CFL limit.
Left: The difference between unity and the eigenvalues of the discretization as a function of the spatial frequency $\omega$ for four different CFL numbers.
The value $1 - |\lambda(\omega)|$ marked with a gold triangle approaches zero as the CFL limit is approached; this eigenvalue is responsible for setting the CFL limit of the discretization.
Right: For $m = 4$, the two-level, MGRIT convergence factor $\max_{\theta} \rho\big( \widehat{{\cal E}}(\omega, \theta) \big)$ from \eqref{eq:rho-omega-theta} as a function of $\omega$ when using the modified coarse-grid operator \eqref{eq:MOL-SL-corrected} for the four CFL numbers used in the left panel.
\label{fig:CFL-limit-issue}
}
\end{figure}

While we do not show plots for the ERK5+U5 discretization analogous to those in Fig.~\ref{fig:CFL-limit-issue}, note that the above discussion also applies to it. 
Notice, however, from Fig.~\ref{fig:rho_vs_CFL-MOL-SL-correct-diss} that the convergence degradation for ERK5+U5 begins at much smaller CFL fractions than for ERK3+U3. 
In essence, this earlier onset occurs because the coarse-grid operator provides a worse approximation to the ideal coarse-grid operator than in the ERK3+U3 case.
It is interesting to note also that the optimized \ptxt{(but not practical)} coarse-grid operator we developed in \cite{DeSterck_etal_2021} resulted in extremely fast convergence for this ERK5+U5 discretization, requiring only three or four iterations to converge for the tests considered there using $c = 0.85 c_{\rm max}$ (see right-hand side of \cite[Tab. 3]{DeSterck_etal_2021}).
\ptxt{
Unlike \eqref{eq:MOL-SL-corrected}, which only mimics the ideal coarse-grid operator for non-dissipative modes that are asymptotically smooth, the eigenvalue-matching optimization strategy in \cite{DeSterck_etal_2021} aims to ensure that the coarse-grid operator mimics the ideal coarse-grid operator over \textit{all} non-dissipative modes, and not only those which are smooth, which is likely why it works at higher CFL numbers than \eqref{eq:MOL-SL-corrected}.
}

Since a set of isolated modes are responsible for causing this significant deterioration in convergence, it is conceivable that wrapping MGRIT with an outer GMRES iteration may be able to restore fast MGRIT convergence when close to the CFL limit.
This approach would be analogous to the efforts in \cite{Oosterlee_Washio_2000} to treat smooth characteristic components in the steady state, spatial multigrid setting.
%

\subsubsection{F- versus FCF-relaxation}
\label{sec:F-vs-FCF}

Up until this point we have used FCF-relaxation in all of our numerical tests.
In our previous works \cite{DeSterck_etal_2021,DeSterck_etal_2022} we found that FCF-relaxation sometimes resulted in substantial improvements in convergence compared to F-relaxation.
This trend also carries over to the setting explored in this paper, as we now detail.

Fig.~\ref{fig:rho_vs_CFL-MOL-SL-correct-diss-F-relax} shows convergence factor plots just as in Fig.~\ref{fig:rho_vs_CFL-MOL-SL-correct-diss} with the key distinction that these correspond to F-relaxation and not FCF-relaxation. (For simplicity, we omit numerically measured convergence factors from Fig.~\ref{fig:rho_vs_CFL-MOL-SL-correct-diss-F-relax}, and plots for the fifth-order discretizations).
Note that the vertical axes of the plots in Fig.~\ref{fig:rho_vs_CFL-MOL-SL-correct-diss-F-relax} are not all the same as those in Fig.~\ref{fig:rho_vs_CFL-MOL-SL-correct-diss}.
It is immediately clear that FCF-relaxation provides substantial improvements in convergence compared to F-relaxation. 
For example, with F-relaxation, there are intervals of CFL numbers for the third-order discretizations for which the convergence factor is larger than unity, while with FCF-relaxation those same convergence factors are bounded uniformly by $0.6$.
%

\begin{figure}[t!]
\centerline{
\includegraphics[scale=0.32]{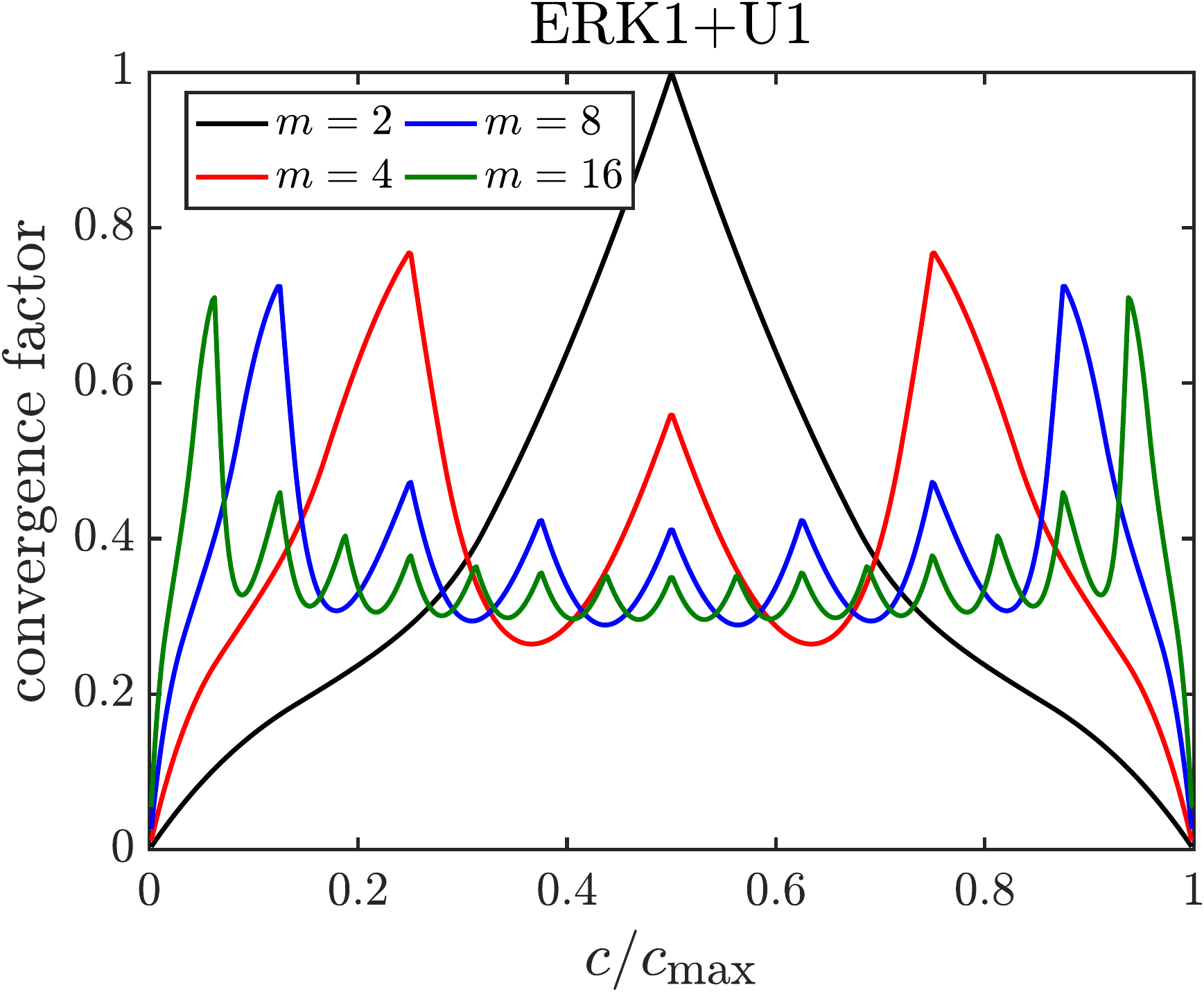}
\quad
\includegraphics[scale=0.32]{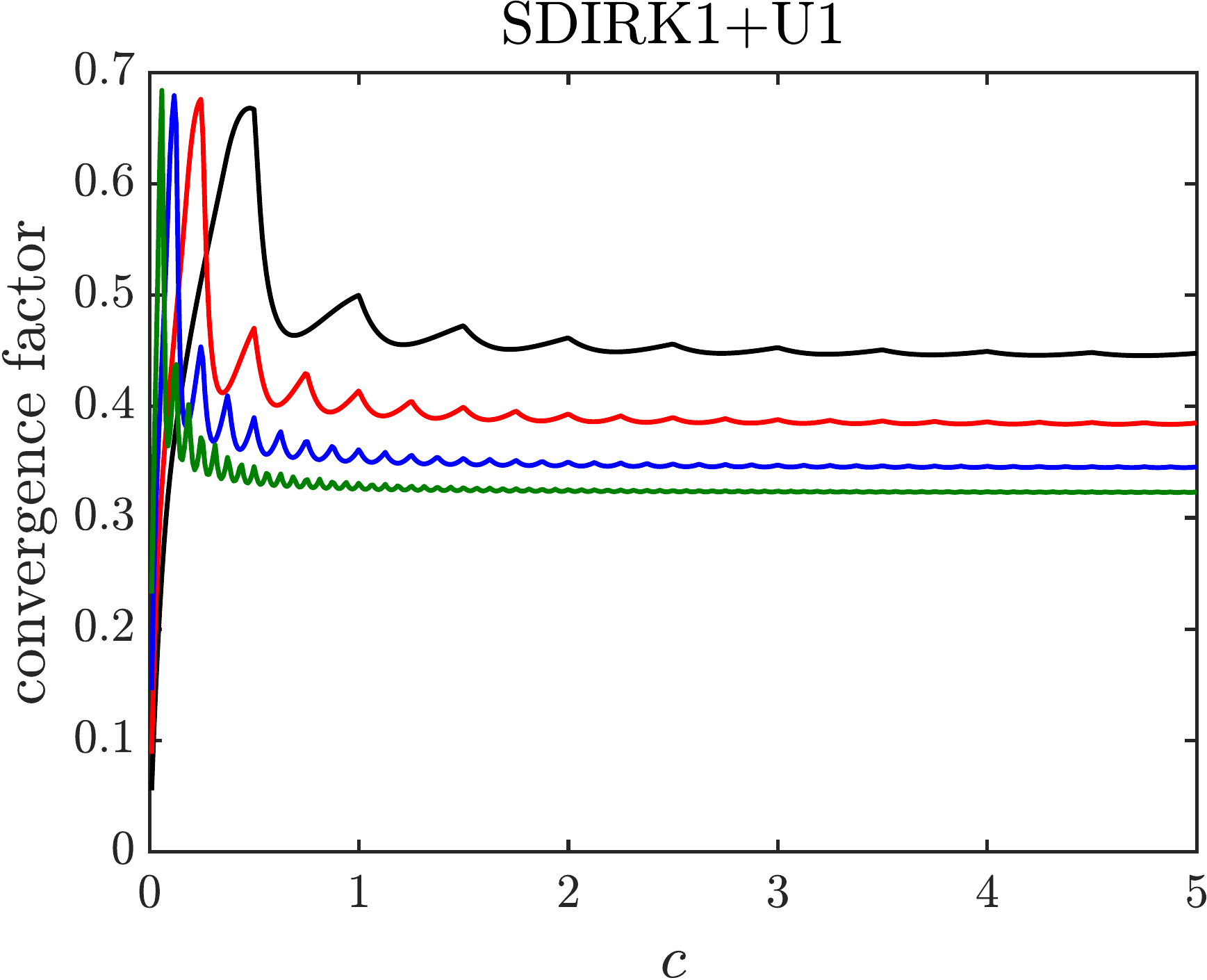}
}
\vspace{2ex}
\centerline{
\includegraphics[scale=0.32]{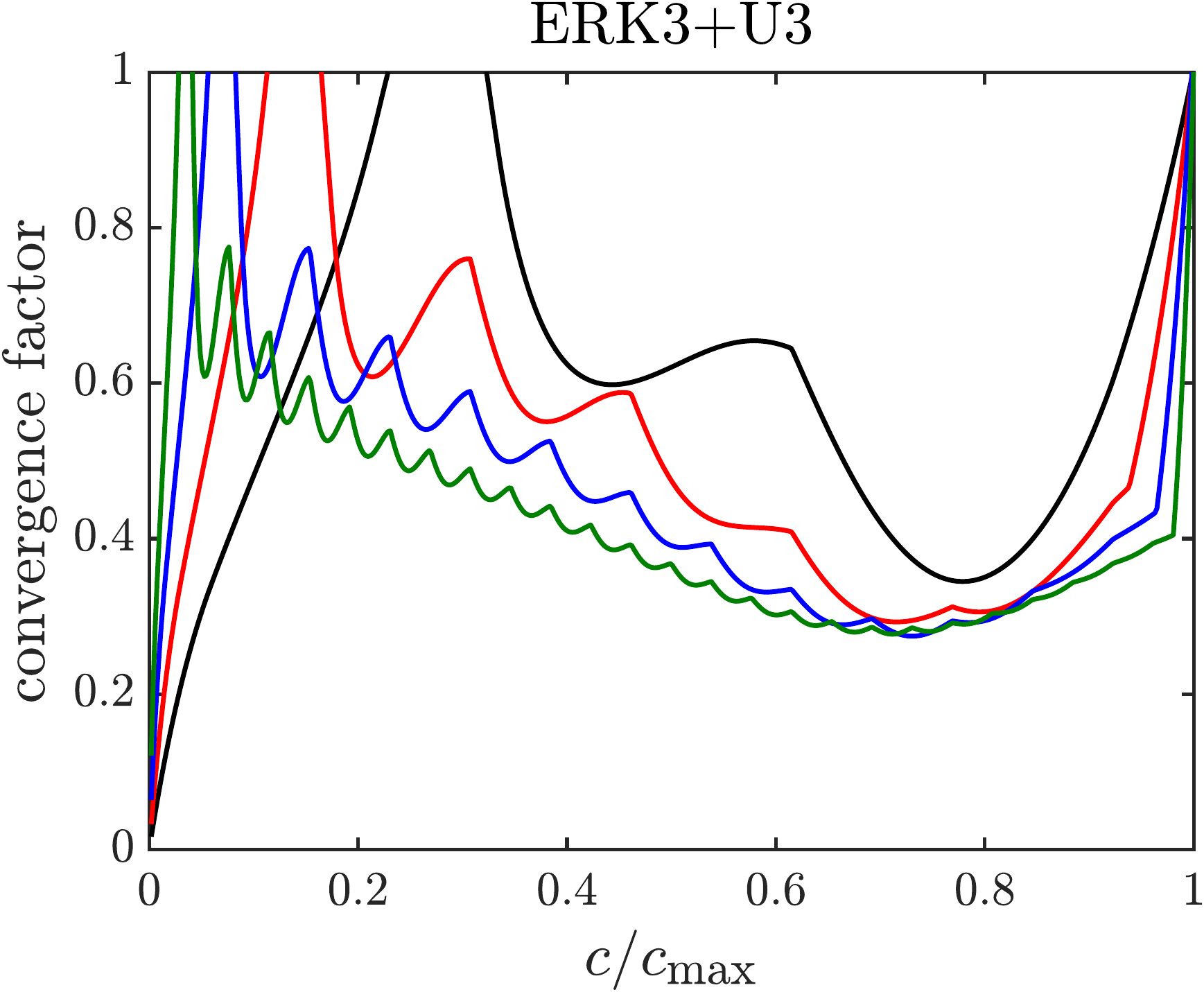}
\quad
\includegraphics[scale=0.32]{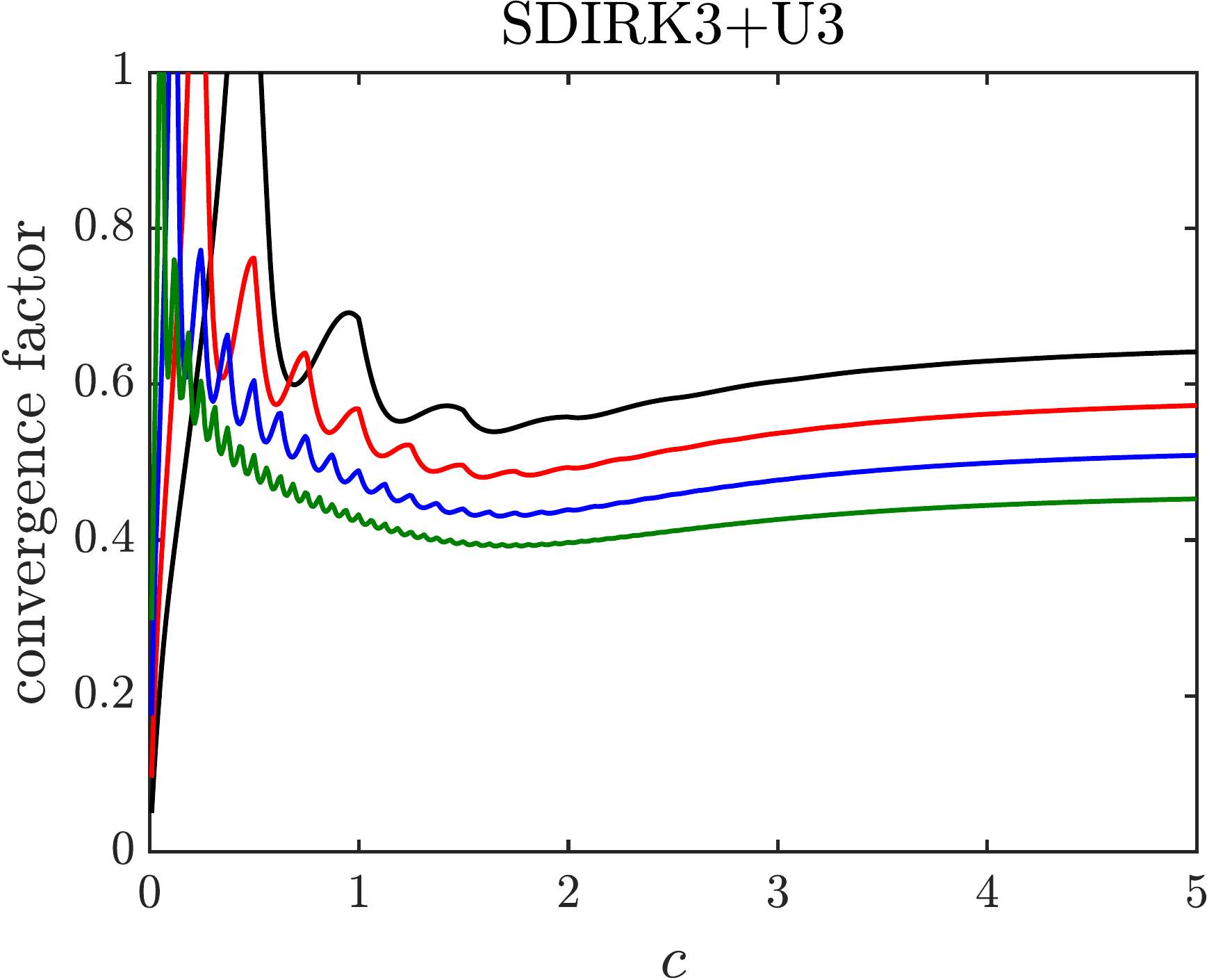}
}
\caption{
Dissipative discretizations with F-relaxation ($\nu = 0$):
Convergence factors predicted by the two-level LFA convergence factor \eqref{eq:rho-asym} for method-of-lines discretizations as a function of the fine-grid CFL number $c$ when using the corrected semi-Lagrangian discretization \eqref{eq:MOL-SL-corrected} on the coarse grid. 
Left: Explicit discretizations. Right: implicit discretizations.
MGRIT uses a coarsening factor of $m$.
Note the different axis scales across the plots.
\label{fig:rho_vs_CFL-MOL-SL-correct-diss-F-relax}
}
\end{figure}

Despite the benefits we show above for FCF-relaxation, other literature has reported little or no benefit for advection-dominated problems when using direct coarse-grid operators \cite{DeSterck_etal_2019,Hessenthaler_etal_2020}. 
Given our analysis for direct coarse-grid operators in Section~\ref{sec:MOL-redisc-analysis}, and specifically Theorem~\ref{thm:MOL-rho-lwr-bnd}, this is not surprising.
That is, in Theorem~\ref{thm:MOL-rho-lwr-bnd} we showed that a direct coarse-grid operator provides an inadequate coarse-grid correction of asymptotically smooth characteristic components, which gives rise to an unsatisfactorily large lower bound on the convergence factor, \textit{independent of the number $\nu$ of CF-relaxations} (recall that asymptotically smooth characteristic components are unaffected by relaxation).
Therefore, if the overall two-grid convergence rate is determined by the convergence of smooth characteristic components, then no improvement in convergence can be obtained by adding relaxation.
In contrast, the reason that additional relaxation benefits the coarse-grid operators presented in this paper is because the two-grid convergence rate of MGRIT is no longer controlled by the coarse-grid correction of the smoothest components, due to the fact that our coarse-grid operator is carefully designed to properly treat these modes.
Finally, we remark that stronger relaxation than FCF (i.e., $\nu \geq 2$) does not appear to lead to any significant improvement over FCF-relaxation in our tests.

\subsection{Numerical results: Dispersive discretizations (even $p$)}
\label{sec:MOL-trunc-correction-num-results-disp}

In this section, we investigate the effectiveness of the coarse-grid operators for dispersive discretizations, in the sense of Definition~\ref{def:diss_vs_disp}.
\ptxt{As in \cite{DeSterck_etal_2022}, we find that our coarse-grid approach does not result in effective MGRIT methods for schemes with even $p$. While we do not yet fully understand why this is the case, it is interesting to look at some results in detail.}

Fig.~\ref{fig:rho_vs_CFL-MOL-SL-correct-disp} presents convergence factor plots as a function of CFL number for second- and fourth-order discretizations. 
\ptxt{MGRIT tests here use the same setup as described in Section \ref{sec:dissipative-convergence}.} 
Unfortunately, convergence is substantially worse than in the case of dissipative discretizations, see Fig.~\ref{fig:rho_vs_CFL-MOL-SL-correct-diss}.
Overall, the convergence factors are much larger than in the dissipative case; in particular, convergence factors surpass one at small CFL numbers for the implicit discretizations, despite fast convergence of implicit discretizations in the dissipative case.
Deterioration in convergence for the explicit schemes as the CFL limit is approached occurs for the same reasons as described in Section~\ref{sec:CFL-limit-deterioration} for the dissipative discretizations.
Compared to the dissipative schemes in Fig.~\ref{fig:rho_vs_CFL-MOL-SL-correct-diss}, the two-level LFA convergence factors in Fig.~\ref{fig:rho_vs_CFL-MOL-SL-correct-disp} do not seem to provide quite as sharp estimates of the effective convergence factors; however, at certain CFL numbers, our numerical tests show some small iteration growth as the mesh is refined, so the estimate may become sharp in the limit of a very fine mesh.
%

\begin{figure}[t!]
\centerline{
\includegraphics[scale=0.32]{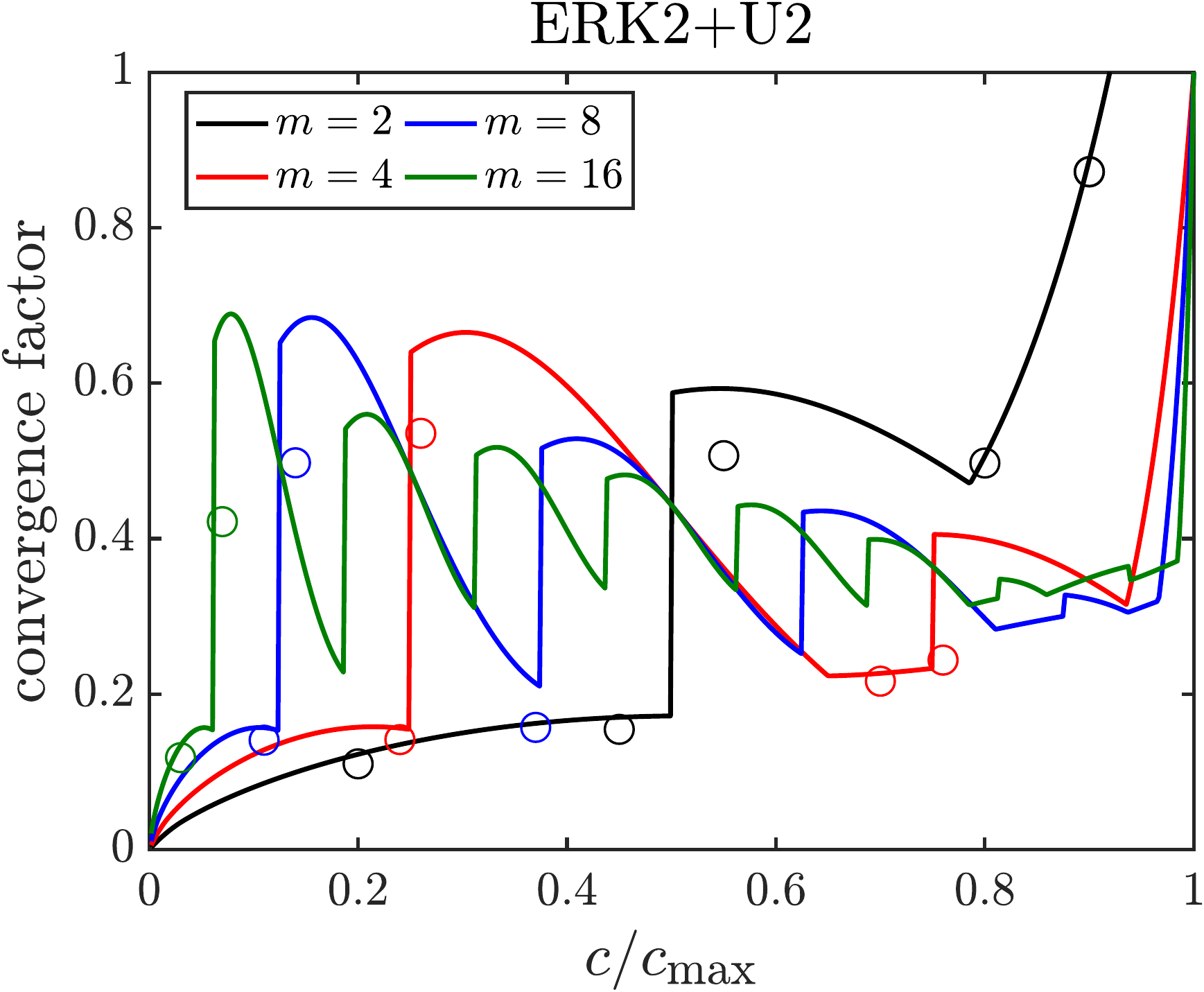}
\quad
\includegraphics[scale=0.32]{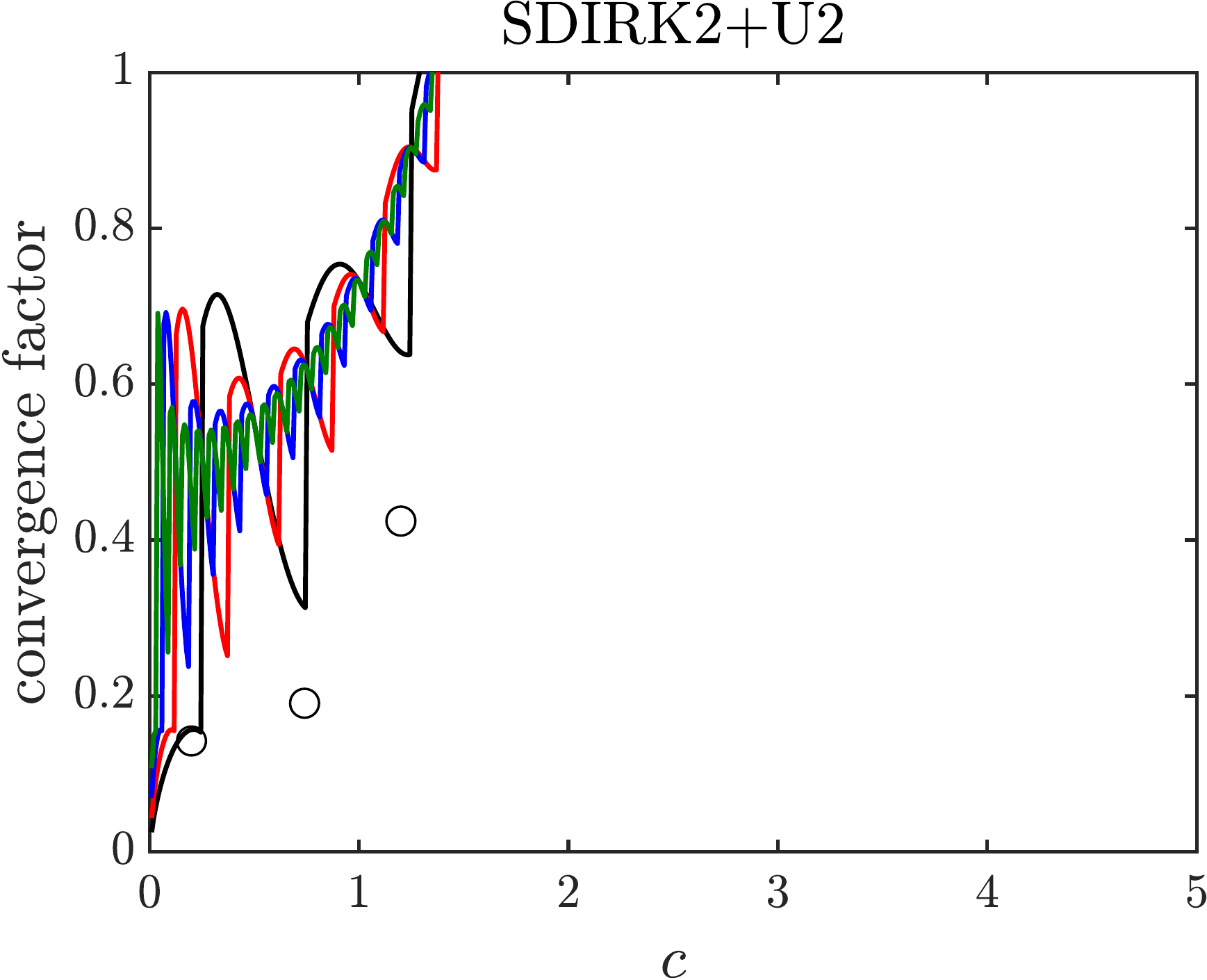}
}
\vspace{2ex}
\centerline{
\includegraphics[scale=0.32]{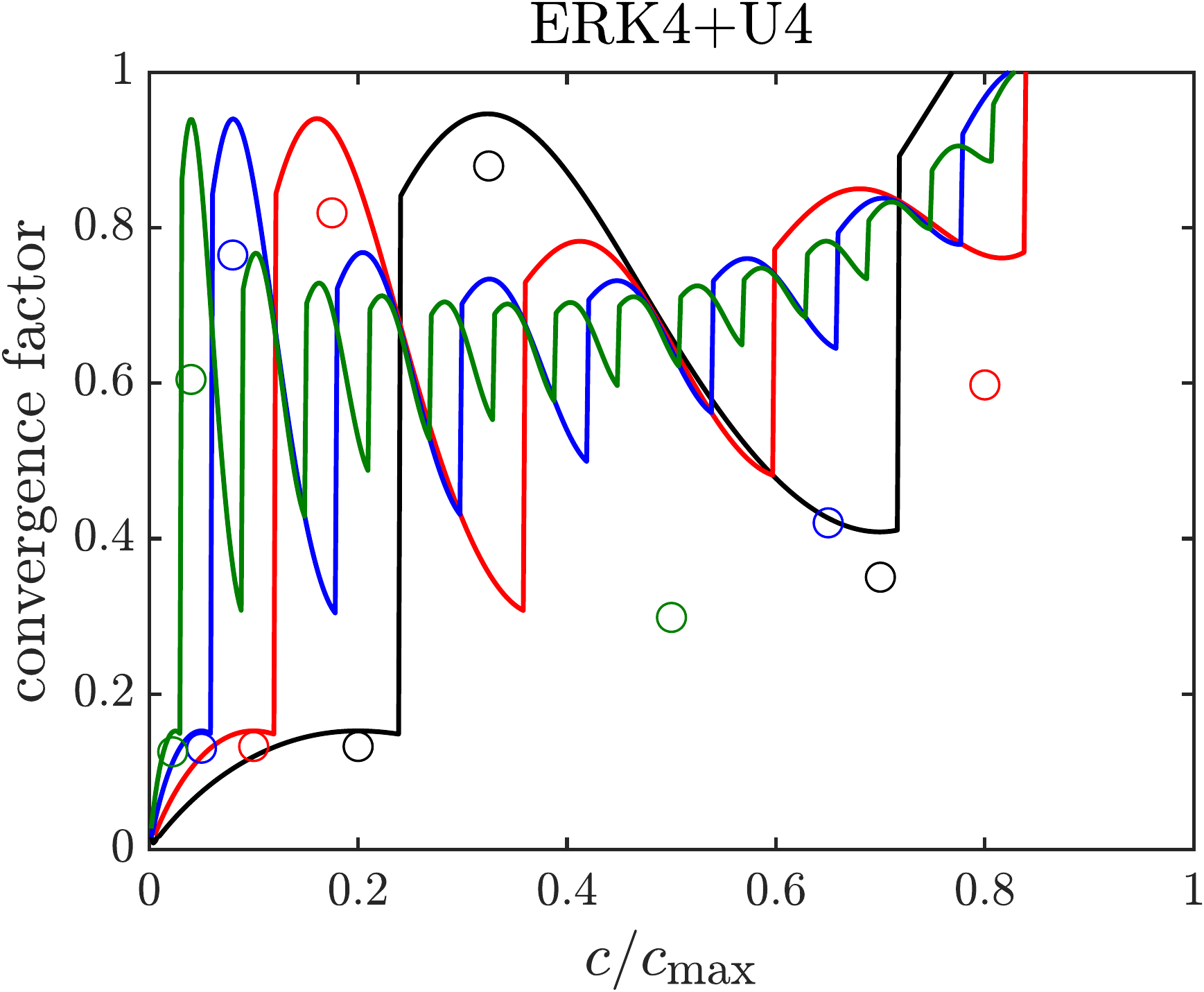}
\quad
\includegraphics[scale=0.32]{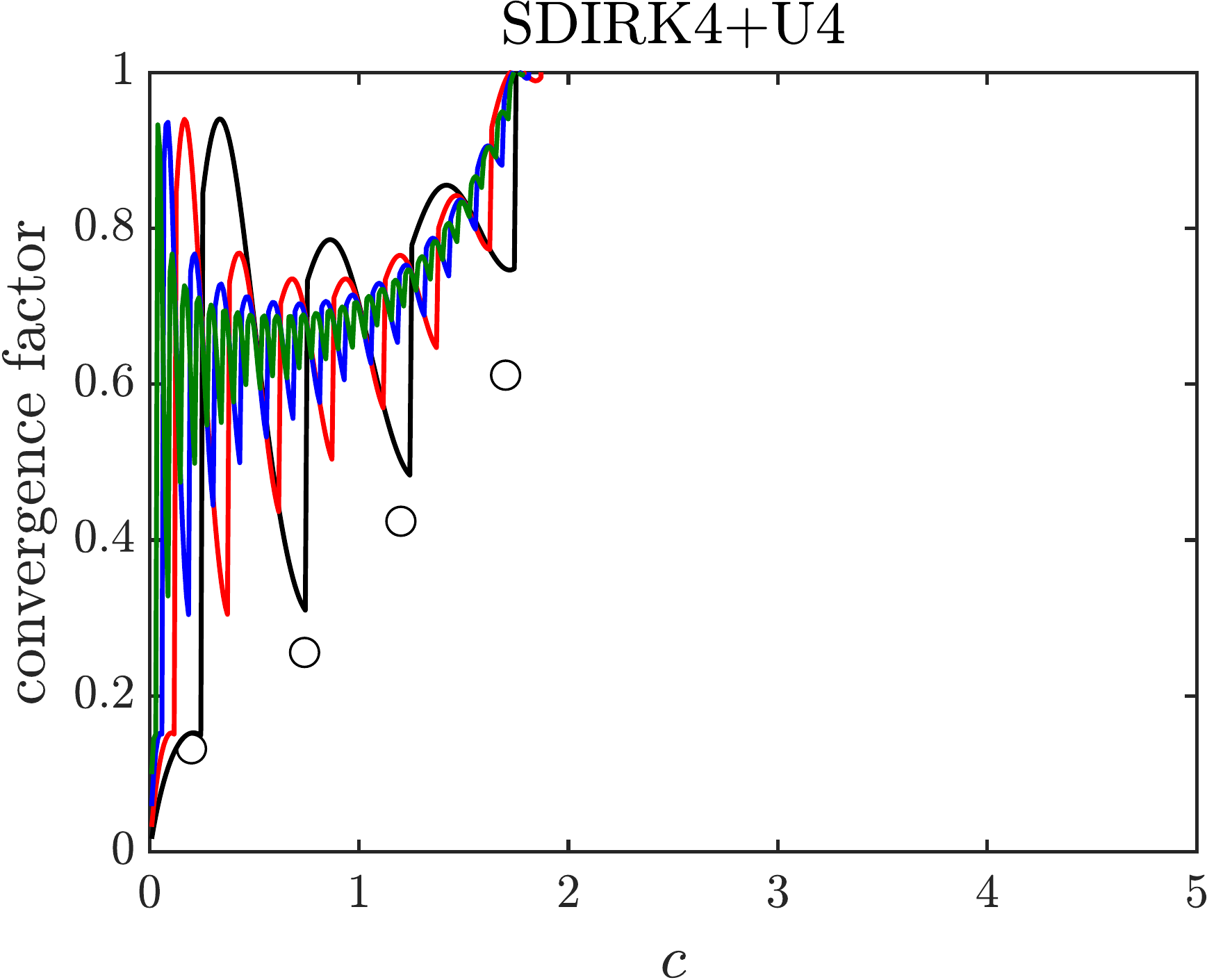}
}
\caption{
Dispersive discretizations:
Convergence factors for method-of-lines discretizations as a function of the fine-grid CFL number $c$ when using the modified semi-Lagrangian discretization \eqref{eq:MOL-SL-corrected} on the coarse grid. 
Left: Explicit discretizations. Right: Implicit discretizations.
MGRIT uses FCF-relaxation ($\nu = 1$), and a coarsening factor of $m$.
Solid lines are the two-level LFA convergence factor \eqref{eq:rho-asym}.
Circle markers are effective two-grid MGRIT convergence factors on the finite interval $t \in (0, T]$.
Note the different axis scales across the plots.
\label{fig:rho_vs_CFL-MOL-SL-correct-disp}
}
\end{figure}

In these tests, the finite-difference matrix ${\cal D}^{(p+1)}_s$ (see \eqref{eq:Dp+1_def}) in the coarse-grid operator \eqref{eq:MOL-SL-corrected} is taken as a first-order accurate discretization \ptxt{(i.e., $s = 1$)}, with a one-point bias to the left.
We use this choice because we find that it gives the best results overall; for example, using skew-symmetric, second-order-accurate discretizations instead lead to catastrophically worse two-level convergence factors that exceed unity at all CFL numbers considered in Fig.~\ref{fig:rho_vs_CFL-MOL-SL-correct-disp}.
%

%
Notice in Fig.~\ref{fig:rho_vs_CFL-MOL-SL-correct-disp} that there are discontinuities in the convergence factor as the CFL number varies, which often results in dramatic changes; for example, when $m = 2$, the two-level convergence factor of ERK4+U4 (bottom left panel of Fig.~\ref{fig:rho_vs_CFL-MOL-SL-correct-disp}) at $c/c_{\rm max} = 0.24^{-}$ is $\approx 0.15$, and at $c/c_{\rm max} = 0.24^{+}$ it is $\approx 0.84$.
These discontinuities arise because of discontinuous shifts in the stencil of the coarse-grid semi-Lagrangian discretization ${\cal S}_p^{(m \delta t)}$ as the coarse-grid CFL number $m c$ is varied.

\ptxt{In \cite{DeSterck_etal_2022}, for the case of dispersive fine-grid semi-Lagrangian discretizations (even $p$),} we similarly found that modified coarse-grid operators akin to \eqref{eq:MOL-SL-corrected} did not yield robust two-grid convergence, with a convergence factor larger than unity for most CFL numbers.
So, from this point of view, it is not surprising that the results in Fig.~\ref{fig:rho_vs_CFL-MOL-SL-correct-disp} are worse than in the dissipative cases shown in Fig.~\ref{fig:rho_vs_CFL-MOL-SL-correct-diss}.
\ptxt{For this reason, we do not consider multilevel tests, and we use a direct solver for the linear systems arising during the application of the coarse-grid operator \eqref{eq:MOL-SL-corrected} so as to show best-case MGRIT convergence.
Moreover, we do not consider parallel scaling studies as we did for dissipative discretizations.
}

\section{Conclusions and future outlook}
\label{sec:conclusion}

\ptxt{
We have considered the MGRIT solution of constant-coefficient linear advection problems discretized with method-of-lines schemes that use upwind finite differences in space and Runge-Kutta methods in time.}
Within the parallel-in-time community, the efficient solution of advection-dominated problems is known to be notoriously difficult, even for such constant-coefficient advection problems.
For example, besides our previous work \cite{DeSterck_etal_2021}, which relied on non-practical computations, we know of no other multigrid-in-time methods that can achieve fast convergence for this class of problems.

In this paper, we have leveraged and built on our earlier works of \cite{DeSterck_etal_2022,DeSterck_etal_2022_LFA}.
Specifically, using the analysis framework we developed in \cite{DeSterck_etal_2022_LFA}, we have proved that directly discretized method-of-lines coarse-grid operators---standard MGRIT coarse-grid operators---fail to provide robust MGRIT convergence with respect to discretization and solver parameters due, at least in part, to their providing an inadequate coarse-grid correction of specific smooth Fourier modes known as characteristic components.
Based on this, we developed modified semi-Lagrangian coarse-grid operators similar to those we developed in \cite{DeSterck_etal_2022} for fine-grid semi-Lagrangian discretizations.
Specifically, the coarse-grid operators we develop here consist of first  applying a semi-Lagrangian discretization followed by a correction term that is carefully designed to alter the truncation error of the resultant operator so that it better approximates that of the ideal coarse-grid operator.
We have shown that these coarse-grid operators are effective for many method-of-lines discretizations of the linear advection problem.

\ptxt{
This work has considered finite-difference spatial discretizations; however, the approach should, in principle, be extendible to other spatial discretizations such as finite volume (FV) and discontinuous Galerkin (DG) methods.
This would be most feasible using (stable) semi-Lagrangian coarse-grid discretizations that use the same type of spatial discretization as the fine grid, and would require error estimates linking the fine- and coarse-grid discretizations.
See \cite{Huang_etal_2012} and \cite{Restelli_etal_2006} (and references therein) for FV- and DG-based semi-Lagrangian schemes, respectively.
}

Despite the progress here, further substantial developments are still required before these techniques can be used to speedup simulations of advection problems in real-world applications.
A first step is the extension of the coarse-grid operators developed here to variable-coefficient linear advection problems.
\ptxt{To this end, we are currently considering heuristically combining the principles developed in this paper and further approximation techniques including optimization. 
We are also exploring extensions to nonlinear hyperbolic PDEs through the use of linearization strategies.}

\appendix

\section{Proof of Lemma~\ref{lem:MOL-trunc}: Method-of-lines truncation error}
\label{app:MOL-trunc-proof}

\begin{proof}
The proof works by substituting the exact PDE solution into a single step of the method-of-lines scheme \eqref{eq:MOL_M_def}; in other words, we evaluate $\bm{u}(t_{n+1})
-
{\cal M}_{p, q}^{(\delta t)}
\bm{u}(t_n)$.
To do this, we require the exact PDE solution $\bm{u}(t)$, and also an expression relating $\bm{u}(t_n)$ to $\bm{u}(t_{n+1})$.

At any $(x,t)$, the exact solution of PDE \eqref{eq:ad} satisfies
\begin{align} \label{eq:u_ad_exp}
u(x, t) = \exp \left( - \alpha t \frac{\partial}{\partial x} \right) u(x, 0),
\end{align}
which can be seen by differentiating both sides of this expression with respect to $t$.
Using \eqref{eq:u_ad_exp} and considering $u(x, t_n) / u(x, t_{n+1})$, we find that
\begin{align} \label{eq:ad_sol:u_new_vs_u_old}
u(x, t_n) = \exp \left( \alpha \delta t \frac{\partial}{\partial x} \right) u(x, t_{n+1}).
\end{align}
Now we replace the spatial partial derivative on the right-hand side of this expression using the finite-difference approximation ${\cal L}_p$. Using the truncation error estimate of ${\cal L}_p$ from Lemma~\ref{lem:FD_trunc}, we may write
\begin{align} \label{eq:MOL_trunc_proof_exp_approx}
\begin{split}
&\left[
\exp \left( \alpha \delta t \frac{\partial}{\partial x} \right)
u(x, t_{n+1})
\right]
\bigg|_{x_i}
\\
&\quad= 
\left( \exp 
\left( 
\alpha \delta t 
\left[
\frac{{\cal L}_p}{h} +
\wh{e}_{\rm FD}  
\frac{{\cal D}^{(p+1)}_s}{h} + {\cal O}(h^{p+1})
\right]
\right) 
\bm{u}(t_{n+1}) \right)_i.
\end{split}
\end{align}
Using \eqref{eq:ad_sol:u_new_vs_u_old} and  \eqref{eq:MOL_trunc_proof_exp_approx}, we see that when the exact PDE solution is substituted into the method-of-lines update \eqref{eq:MOL_M_def} it produces the residual
\begin{align} \label{eq:MOL_trunc_proof_residual}
\begin{split}
&\bm{u}(t_{n+1})
-
{\cal M}_{p, q}^{(\delta t)}
\bm{u}(t_n)
\\
&\quad=
\left\{
I
-
{\cal M}_{p, q}^{(\delta t)} 
\exp 
\left( 
\alpha \delta t 
\left[
\frac{{\cal L}_p}{h} +
\wh{e}_{\rm FD}  
\frac{{\cal D}^{(p+1)}_s}{h} + {\cal O}(h^{p+1})
\right]
\right) 
\right\}
\bm{u}(t_{n+1}).
\end{split}
\end{align}
We now simplify the right-hand side of this equation.

First we simplify the method-of-lines time-stepping operator defined in \eqref{eq:MOL_M_def}, ${\cal M}_{p, q}^{(\delta t)} = R_q\left( 
- \frac{\alpha \delta t  {\cal L}_p}{h}
\right)$.
Using the expansion of $R_q$ given in \eqref{eq:e_RK_def}, for any sufficiently smooth vector $\bm{v} \in \mathbb{R}^{n_x}$, we have
\begin{align}
{\cal M}_{p, q}^{(\delta t)}  \bm{v}
&=
\exp 
\left( 
- \alpha \delta t  \frac{{\cal L}_p}{h}
\right)
\bm{v}
+
\eRK
\left(-\alpha \delta t \frac{{\cal L}_p}{h} \right)^{q+1}
\bm{v}
+
{\cal O} \Big( \delta t^{q+2} \Big),
\\
\label{eq:MOL_trunc_proof_Mpq}
&=
\exp 
\left( 
- \alpha \delta t  \frac{ {\cal L}_p}{h}
\right)
\bm{v}
+
\eRK
(- \alpha \delta t)^{q+1} \frac{ {\cal D}^{(q+1)}_s}{ h^{q+1} } 
\bm{v}
+
{\cal O} \Big( h^p \delta t^{q+1}, \delta t^{q+2} \Big).
\end{align}
Note that this second expression follows from the fact that $\big( \frac{{\cal L}_p}{h} \big)^{q+1} \bm{v}$ is a $p$th-order approximation to the $q+1$st derivative of $\bm{v}$, and, thus, we may write $\big( \frac{{\cal L}_p}{h} \big)^{q+1} \bm{v} = \frac{{\cal D}^{(q+1)}_s}{h^{q+1}} \bm{v} + {\cal O}(h)$ by recalling that the approximation $\frac{{\cal D}^{(q+1)}_s}{h^{q+1}}$ must be at least first-order accurate (i.e., $s \geq 1$ in \eqref{eq:Dp+1_def}).
Plugging \eqref{eq:MOL_trunc_proof_Mpq} into the right-hand side of \eqref{eq:MOL_trunc_proof_residual} and simplifying gives (omitting the $\bm{u}(t_{n+1})$)
\begingroup 
\allowdisplaybreaks
\begin{align}
&I
-
{\cal M}_{p, q}^{(\delta t)} 
\exp 
\left( 
\alpha \delta t 
\left[
\frac{{\cal L}_p}{h} +
\wh{e}_{\rm FD}  
\frac{{\cal D}^{(p+1)}_s}{h} + {\cal O}(h^{p+1})
\right]
\right)
\\ 
\begin{split}
&=
I
-
\bigg[
\exp 
\left( 
- \alpha \delta t \frac{ {\cal L}_p}{h}
\right)
+
\eRK
(- \alpha \delta t)^{q+1} \frac{ {\cal D}^{(q+1)}_s}{ h^{q+1} } 
+
{\cal O} \Big( h \delta t^{q+1}, \delta t^{q+2} \Big)
\bigg]
\\ 
&\quad\quad
\times
\exp 
\left( 
\alpha \delta t 
\left[
\frac{{\cal L}_p}{h} +
\wh{e}_{\rm FD}  
\frac{{\cal D}^{(p+1)}_s}{h} + {\cal O}(h^{p+1})
\right]
\right), 
\end{split}
\\ 
\begin{split}
&=
I
-
\exp 
\left( 
\alpha \delta t 
\left[
\wh{e}_{\rm FD}  
\frac{{\cal D}^{(p+1)}_s}{h} + {\cal O}(h^{p+1})
\right]
\right) 
\\ 
&\quad\quad
-
\left\{
\eRK
(- \alpha \delta t)^{q+1} \frac{ {\cal D}^{(q+1)}_s}{ h^{q+1} } 
+
{\cal O} \Big( h \delta t^{q+1}, \delta t^{q+2} \Big)
\right\}
\Big\{
I + {\cal O}(\delta t)
\Big\},
\end{split}
\\ 
\label{eq:MOL_trunc_proof_residual_operator}
\begin{split}
&=
-
\frac{\alpha \delta t}{h} 
\wh{e}_{\rm FD} 
h^{p+1} 
\frac{{\cal D}^{(p+1)}_s}{h^{p+1}} 
\\ 
&\quad\quad
-
\eRK
\left(- \frac{\alpha \delta t}{h} \right)^{q+1} 
h^{q+1} 
\frac{ {\cal D}^{(q+1)}_s}{ h^{q+1} } 
+
{\cal O} \Big(h^{p+1} \delta t,  h \delta t^{q+1}, \delta t^{q+2} \Big).
\end{split}
\end{align}
\endgroup
Plugging \eqref{eq:MOL_trunc_proof_residual_operator} into the right-hand side of \eqref{eq:MOL_trunc_proof_residual} gives the claim \eqref{eq:Mpq_trunc}.
\qed
\end{proof}
%

\section{Proof of Lemma~\ref{lem:Psi_ideal-MOL-pert}: Ideal coarse-grid operator perturbation}
\label{app:Psi_ideal_MOL-pert-proof}
\begin{proof}
We begin by developing an estimate for the action of the ideal coarse-grid operator.
Rearranging the method-of-lines truncation error result \eqref{eq:Mpq_trunc} from Lemma~\ref{lem:MOL-trunc} gives
\begin{align} \label{eq:Mpp_trunc-rearranged}
\begin{aligned}
{\cal M}^{(\delta t)}_{p, p} \bm{u}(t_n) 
&=
{\cal O} \Big(
h^{p+1} \delta t,
h \delta t^{p+1}, 
\delta t^{p+2} 
\Big)
+
\left(
I 
+ 
\left[
c  
\mkern 1mu 
\eFD
+
(-c)^{p+1} \eRK
\right] 
{\cal D}^{(p+1)}_s
\right)
\bm{u}(t_{n+1}).
\end{aligned}
\end{align}
Now multiply both sides of this expression by ${\cal M}^{(\delta t)}_{p, p}$.
To simplify the result, first use the fact that ${\cal M}^{(\delta t)}_{p, p}$ commutes with ${\cal D}^{(p+1)}_s$ (they are both circulant matrices), and then use \eqref{eq:Mpp_trunc-rearranged} to estimate ${\cal M}^{(\delta t)}_{p, p} \bm{u}(t_{n+1})$. 
This gives
\begingroup 
\allowdisplaybreaks
\begin{align} 
&{\cal M}^{(\delta t)}_{p, p} {\cal M}^{(\delta t)}_{p, p} \bm{u}(t_n) 
=
\left(
I 
+ 
\left[
c  
\mkern 1mu 
\eFD
+
(-c)^{p+1} \eRK
\right] 
{\cal D}^{(p+1)}_s
\right)
{\cal M}^{(\delta t)}_{p, p}
\bm{u}(t_{n+1})
\\ \notag 
&\hspace{12em} 
+
{\cal O} \Big(
h^{p+1} \delta t,
h \delta t^{p+1}, 
\delta t^{p+2} 
\Big),
\\ 
&\quad=
\left(
I 
+ 
\left[
c  
\mkern 1mu 
\eFD
+
(-c)^{p+1} \eRK
\right] 
{\cal D}^{(p+1)}_s
\right)^2
\bm{u}(t_{n+2})
+
{\cal O} \Big(
h^{p+1} \delta t,
h \delta t^{p+1}, 
\delta t^{p+2} 
\Big),
\\ 
\label{eq:Psi_ideal-MOL-pert-aux}
&\quad=
\left(
I 
+
2 
\left[
c  
\mkern 1mu 
\eFD
+
(-c)^{p+1} \eRK
\right] 
{\cal D}^{(p+1)}_s
\right)
\bm{u}(t_{n+2})
+
{\cal O} \Big(
h^{2(p+1)},
h^{p+1} \delta t,
h \delta t^{p+1}, 
\delta t^{p+2} 
\Big).
\end{align}
\endgroup
Note that \eqref{eq:Psi_ideal-MOL-pert-aux} follows by discarding the highest-order term when expanding the square, since \\${\cal D}^{(p+1)}_s {\cal D}^{(p+1)}_s \bm{u}(t_{n+2}) = {\cal O}(h^{2(p+1)})$.
Applying ${\cal M}^{(\delta t)}_{p, p}$ to both sides of \eqref{eq:Psi_ideal-MOL-pert-aux} a further $m - 2$ times, each time applying the above arguments, by induction we arrive at the following estimate for the action of the ideal coarse-grid operator
\begin{align} \label{eq:Psi_ideal-MOL-pert-Psi_ideal_est}
\left[ \prod \limits_{k = 0}^{m-1} {\cal M}^{(\delta t)}_{p, p} \right] 
\bm{u}(t_n) 
&=
\Psi_{\rm ideal} 
\bm{u}(t_n) 
=
{\cal O} \Big(
h^{2(p+1)},
h^{p+1} \delta t,
h \delta t^{p+1}, 
\delta t^{p+2} 
\Big)
\\ \notag 
&\hspace{-2em} 
+
\left(
I 
+
m
\left[
c  
\mkern 1mu 
\eFD
+
(-c)^{p+1} \eRK
\right] 
{\cal D}^{(p+1)}_s
\right)
\bm{u}(t_{n+m}).
\end{align}

Noticing the $\bm{u}(t_{n+m})$ on the right-hand side of \eqref{eq:Psi_ideal-MOL-pert-Psi_ideal_est}, we seek to rewrite this in terms of the action of the  semi-Lagrangian operator, ${\cal S}_{p}^{(m \delta t)}$.
To this end, rewriting the semi-Lagrangian truncation error estimate \eqref{eq:SL_trunc} from Lemma~\ref{lem:SL_trunc} we have
\begin{align} \label{eq:SL_trunc-rewritten}
\left(
I - (-h)^{p+1} 
f_{p+1} \big(\varepsilon^{(\delta t)} \big) 
{\cal D}^{(p+1)}_s
\right)
\bm{u}(t_{n+m})
=
{\cal S}_{p}^{(\delta t)} \bm{u}(t_{n})
+ {\cal O}(h^{p+2}).
\end{align}
Rearranging for $\bm{u}(t_{n+m})$, employing a geometric expansion valid for small $h$, and then using the fact that $\bm{u}(t_{n})$ is smooth, we have
\begin{align} 
\bm{u}(t_{n+m})
&=
\left(
I - (-h)^{p+1} 
f_{p+1} \big(\varepsilon^{(\delta t)} \big) 
{\cal D}^{(p+1)}_s
\right)^{-1}
{\cal S}_{p}^{(\delta t)} \bm{u}(t_{n})
+ {\cal O}(h^{p+2}),
\\
&=
\Big(
I 
+ 
(-h)^{p+1} 
f_{p+1} \big(\varepsilon^{(\delta t)} \big) 
{\cal D}^{(p+1)}_s
+
\big[(-h)^{p+1} 
f_{p+1} \big(\varepsilon^{(\delta t)} \big) 
{\cal D}^{(p+1)}_s
\big]^2
\\
\notag
&\quad\quad\quad
+
{\cal O} \Big( \big[(-h)^{p+1} 
f_{p+1} \big(\varepsilon^{(\delta t)} \big) 
{\cal D}^{(p+1)}_s
\big]^3 \Big)
\Big)
{\cal S}_{p}^{(\delta t)} \bm{u}(t_{n})
+ {\cal O}(h^{p+2}),
\\
\label{eq:SL_trunc-expansion}
&=
\left(
I + (-h)^{p+1} 
f_{p+1} \big(\varepsilon^{(\delta t)} \big) 
{\cal D}^{(p+1)}_s
\right)
{\cal S}_{p}^{(\delta t)} \bm{u}(t_{n})
+ {\cal O}(h^{p+2}).
\end{align}

Next, plug the expression for $\bm{u}(t_{n+m})$ given by \eqref{eq:SL_trunc-expansion} into the right-hand side of \eqref{eq:Psi_ideal-MOL-pert-Psi_ideal_est} to give
\begin{align}
&\Psi_{\rm ideal} 
\bm{u}(t_n) 
=
\left(
I 
+
m
\left[
c  
\mkern 1mu 
\eFD
+
(-c)^{p+1} \eRK
\right] 
{\cal D}^{(p+1)}_s
\right)
\\ \notag 
&\hspace{1em} 
\times \left(
I + (-h)^{p+1} 
f_{p+1} \big(\varepsilon^{(\delta t)} \big) 
{\cal D}^{(p+1)}_s
\right)
{\cal S}_{p}^{(\delta t)} \bm{u}(t_{n})
+
{\cal O} \Big(
h^{p+2},
h^{p+1} \delta t,
h \delta t^{p+1}, 
\delta t^{p+2} 
\Big),
\\ 
\label{eq:Psi_ideal-MOL-pert-aux2}
&=
\left(
I 
+
\left[
m c  
\mkern 1mu 
\eFD
+
m (-c)^{p+1} \eRK
+
(-h)^{p+1} 
f_{p+1} \big(\varepsilon^{(\delta t)} \big) 
\right] 
{\cal D}^{(p+1)}_s
\right)
{\cal S}_{p}^{(\delta t)} \bm{u}(t_{n})
\\ \notag 
&\hspace{0em} 
+
{\cal O} \Big(
h^{p+2},
h^{p+1} \delta t,
h \delta t^{p+1}, 
\delta t^{p+2} 
\Big).
\end{align}
This proves the first claim \eqref{eq:Psi_ideal-SL_redisc-pert0} of the lemma, noting that the constant multiplying ${\cal D}^{(p+1)}_s$ in \eqref{eq:Psi_ideal-MOL-pert-aux2} is indeed $\varphi_{p+1}^{(m \delta t)}$ given in \eqref{eq:varphi-constant-MOL-SL}.

The second claim \eqref{eq:Psi_ideal-SL_redisc-pert} of the lemma follows from \eqref{eq:Psi_ideal-SL_redisc-pert0} by a geometric expansion akin to that used in \eqref{eq:SL_trunc-expansion}, 
$ \big( I - \varphi^{(m \delta t)}_{p+1} {\cal D}^{(p+1)}_s \big)^{-1} {\cal S}_p^{(m \delta t)} \bm{u}(t_n) 
= 
\big( I + \varphi^{(m \delta t)}_{p+1} {\cal D}^{(p+1)}_s \big) {\cal S}_p^{(m \delta t)} \bm{u}(t_n)
+
{\cal O}(h^{2(p+1)})
$.
\qed
\end{proof}

\begin{acknowledgements}
We are grateful for feedback from anonymous referees which helped improve the quality of this manuscript.
\end{acknowledgements}

\bibliographystyle{spmpsci}      
\bibliography{pit-advection-MOL-generalization-R1}

\section*{Statements \& Declarations}

\subsection*{Funding}
This work was performed under the auspices of the U.S. Department of Energy by Lawrence Livermore National Laboratory under Contract DE-AC52-07NA27344 (LLNL-JRNL-839789).
This work was supported in part by the U.S. Department of Energy, Office of Science, Office of Advanced Scientific Computing Research, Applied Mathematics program, and by NSERC of Canada.

%
\subsection*{Competing Interests}
The authors declare that they have no conflict of interest.

\subsection*{Author Contributions}

\subsection*{Data Availability}
Data sharing not applicable to this article as no datasets were generated on analyzed during the current study.

\end{document}